\documentclass[a4paper,11pt]{amsart}
\usepackage{mathrsfs}
 \usepackage[all]{xy}
\usepackage{amsmath,amssymb,amscd,bbm,amsthm,mathrsfs}
\newtheorem{thm}{Theorem}[section]
\newtheorem{lem}{Lemma}[section]
\newtheorem{cor}{Corollary}[section]
\newtheorem{prop}{Proposition}[section]
\newtheorem{rem}{Remark}[section]

\begin{document}

 \title[Anti-self-dual connections over $5$D Heisenberg group and the twistor method]{
Anti-self-dual connections over the $5$D Heisenberg group and the twistor method}

\begin{abstract}
In this paper, we introduce   notions of   $\alpha$-planes in  $5$D complex Heisenberg group and the twistor space as the moduli space of all $\alpha$-planes. So we can define an anti-self-dual (ASD) connection as a connection flat over all $\alpha$-planes. This geometric approach allows us to establish Penrose-Ward correspondence between ASD connections over $5$D complex Heisenberg group and a class of holomorphic vector bundles on the twistor space. By Atiyah-Ward ans\"{a}tz, we also construct a family of ASD connections on  $5$D complex Heisenberg group. When restricted to $5$D real Heisenberg group, the flat model of $5$D contact manifolds, an ASD connection satisfies the horizontal part of the contact instanton equation  introduced by physicists.
\end{abstract}
\author{Guangzhen\ Ren}

\address{Department of Mathematics, Zhejiang International Studies University, Hangzhou 310023, PR China}

\email [G.-Z.\ Ren]{gzren@zisu.edu.cn}

\author{Wei\ Wang}

\address{Department of Mathematics, Zhejiang University, Zhejiang 310027, PR China}

\email[W. Wang]{wwang@zju.edu.cn}

\thanks{The first author is partially supported by National Nature Science Foundation in China (No.
11801508, 11801523, 11971425); The second author is partially supported by National Nature Science Foundation in China (No.
11971425)}

 \maketitle

\section{Introduction}
Since $3$D Chern-Simons theory is very successful, people are interested in its generalization to higher odd dimension. When studying $5$D Chern-Simons theory, physicists
introduced the supersymmetric Yang-Mills equation on $5$D contact manifolds (c.f. \cite{MR2967134} \cite{MR3042942}). In particular, they introduced the {\it contact instanton equation} on $5$D contact manifolds as
\begin{equation}
\mathcal{F}=\pm\iota_{T}\ast\mathcal{F},
\end{equation}
where $\mathcal{F}$ is the curvature for a connection, $T$ is the Reeb vector field and $\iota_{T}$ is the contraction with $T$. Hosomichi-Seong-Terashima \cite{MR2967134} constructed a $5$D $\mathcal{N}=1$ supersymmetric Yang-Mills theory on the five sphere, and   showed   the fields in a vector multiplet localized to contact instantons by using the localization technique.
 K\"{a}ll\'{e}n-Zabzine \cite{MR3042942} localized the path integral of Chern-Simons theory on    circle fibrations over   $4$D sympletic manifolds to contact instantons. On the other hand,
Itoh \cite{Itoh2002Contact} has constructed the CR twistor space over $5$D contact manifolds by using differential geometry about two decades ago.   Wolf \cite{Wolf1} used the twistor method to study the contact instanton equation on contact manifolds.

In the classical flat case, by complexifying $4$D Minkowski space as $\mathbb{C}^{4}$, ones can construct twistor space by using complex geometry method (cf. \cite{MR2583958} \cite{Mason1} \cite{MR1054377}). If we denote a point of $\mathbb{C}^{4}$ by
 $$\mathbf{y}=\left(\begin{array}{lll}
y_{00'}  &y_{01'}\\
y_{10'} &y_{11'}
\end{array}\right),$$
 an {\it $\alpha$-plane} in $\mathbb{C}^{4}$ is the set of all $\mathbf{y}$ satisfying
\begin{equation} \left(\begin{array}{lll}
y_{00'}  &y_{01'}\\
y_{10'} &y_{11'}
\end{array}\right)
 \left(\begin{array}{ll}
\pi_{0'}\\
\pi_{1'}
\end{array}\right)=
 \left(\begin{array}{ll}
\omega_{0}\\
\omega_{1}
\end{array}\right)
  \end{equation}
for fixed $0\neq(\pi_{0'},\pi_{1'})\in\mathbb{C}^{2}$ and $w_{0},$ $w_{1}$ $\in\mathbb{C}$.
The moduli space of all
$\alpha$-planes is the twistor space $\mathcal{P}_{0}$, which is an open subset of $\mathbb{C}P^3$. Then we have the double fibration
\begin{equation*}        \xy 0;/r.22pc/:
 (0,20)*+{\mathbb{C}^{4}\times \mathbb{C}P^1}="1";
(-20,0)*+{  \mathcal{P}_{0} }="2";
(20,0)*+{\mathbb{C}^{4}}="3";
{\ar@{->}_{\eta}"1"; "2"};
{\ar@{->}^{\tau} "1"; "3"};
\endxy,
 \end{equation*}
based on which there exists the Penrose correspondence between the solution of massless field equations and the first cohomology groups of certain line bundles over $\mathcal{P}_{0}$, and Penrose-Ward correspondence between the solutions of ASD Yang-Mills equation and holomorphic bundles over $\mathcal{P}_{0}$, trivial over the complex projective line $\hat{x}=\eta\circ\tau^{-1}(x)$ for any $x\in\mathbb{C}^{4}.$ In this paper, we consider the simplest $5$D contact manifold, $5$D real Heisenberg group, and complexify it as $\mathbb{C}^{5}$. We will generalize this theory to   $5$D Heisenberg group.

$5$D complex {\it  Heisenberg  group}   $\mathscr H$  is
$\mathbb{C}^{5}:=\{(\mathbf{y},t)|\mathbf{y}\in\mathbb{C}^{4},t\in\mathbb{C}\}$ with the multiplication given by
\begin{equation} \label{eq:w-multiplication} \begin{split}&(\mathbf{y},t ) \circ
(\mathbf{y'},t')=\left ( \mathbf{y}+\mathbf{y'},t+t' +
B(\mathbf{ y}, \mathbf{y'})\right),
\end{split}\end{equation}
where $B(\mathbf{y}, \mathbf{y'})=y_{00'}y'_{11'}-y_{01'}y'_{10'}+y_{10'}y'_{01'}
-y_{11'}y'_{00'}.$ We have left invariant vector fields on $ \mathscr H$:
\begin{equation} \label{eq:Y}\begin{split}
V_{00'}&:=\frac{\partial}{\partial y_{00'}}- {y}_{11'} T,
\qquad V_{01'}:=\frac{\partial}{\partial y_{01'}}+ {y}_{10'} T,\\
V_{10'}&:=\frac{\partial}{\partial y_{10'}}- {y}_{01'} T,\qquad
V_{11'}:=\frac{\partial}{\partial y_{11'}}+{y}_{00'} T,\qquad
T:=\frac{\partial}{\partial t}.
\end{split}
\end{equation}
 It is easy to see that
\begin{equation} \label{eq:YYT}
[V_{00'},V_{11'} ]=[V_{10'},V_{01'} ]=2T,
\end{equation}
and all other brackets vanish. Consequently, for fixed $0\neq(\pi_{0'},\pi_{1'})\in\mathbb{C}^{2}$, if denote
\begin{equation}\label{keyvf}V_{A}:=\pi_{0'}V_{A0'}-\pi_{1'}V_{A1'},\quad A=0,1,\end{equation}
we have $$[V_{0},V_{1}]=0.$$ Namely, $span\{V_{0},V_{1}\}$ is an abelian Lie subalgebra and an integrable distribution for fixed $0\neq(\pi_{0'},\pi_{1'})\in\mathbb{C}^{2}$.
Their integral surfaces are hyperplanes (cf. (\ref{eq:lienar eq0})), which we also call {\it $\alpha$-planes}. The {\it twistor space} $\mathcal{P}$ is the moduli space of all $\alpha$-planes, which is a $4$D complex manifold.
We have the double fibration over $5$D complex Heisenberg group as follows
\begin{equation} \label{TOH}       \xy 0;/r.22pc/:
 (0,20)*+{\mathcal{F}=\mathbb{C}^{5}\times \mathbb{C}P^1 }="1";
(-20,0)*+{  \mathcal{P} }="2";
(20,0)*+{\mathscr H\cong \mathbb{C}^{5}}="3";
{\ar@{->}_{\eta}"1"; "2"};
{\ar@{->}^{\tau} "1"; "3"};
\endxy.
 \end{equation}

A connection is called {\it anti-self-dual} (briefly ASD) if it is flat over any $\alpha$-plane.
Let $\Phi=\Phi_{00'}\theta^{00'}+\Phi_{10'}\theta^{10'}+\Phi_{01'}\theta^{01'}+\Phi_{11'}\theta^{11'}+\Phi_{T}\theta$ be a
$\mathfrak g$-valued connection form on $\mathscr H$, where $\{\theta^{AA'},\theta\}$ are $1$-forms dual to $\{V_{AA'},T\}$.
$\Phi$ is ASD if and only if it satisfies the {\it ASD Yang-Mills equation}
\begin{equation}\label{eq:ASD}
     \left\{\begin{array}{l}
V_{00'}(\Phi_{10'})-V_{10'}(\Phi_{00'})+[\Phi_{00'},\Phi_{10'}]=0,\\
V_{01'}(\Phi_{10'})+V_{00'}(\Phi_{11'})-V_{10'}(\Phi_{01'})-V_{11'}(\Phi_{00'})+[\Phi_{00'},\Phi_{11'}]+[\Phi_{01'},\Phi_{10'}]=0,\\
V_{01'}(\Phi_{11'})-V_{11'}(\Phi_{01'})+[\Phi_{01'},\Phi_{11'}]=0.
         \end{array}   \right.
  \end{equation}

An open subset $U$ of $\mathscr H$   is called {\it elementary} if for every $\alpha$-plane $\widetilde{Z}=\tau\circ\eta^{-1}(Z)$ ($Z\in \mathcal{P} $)  the intersection $\widetilde{Z}\cap U$ is connected and simply connected. Then we have Penrose-Ward  correspondence.
\begin{thm}\label{pwc} Let $U$ be an elementary open set in $\mathscr H$. There is a one-to-one correspondence between\\
 (1) gauge equivalence classes of $ASD$ connections with the gauge group ${\rm GL}(n,\mathbb{C})$ over $U$; \\
 (2) holomorphic
vector bundles $E'\rightarrow \hat{U}=\eta\circ\tau^{-1}(U)$ such that $E'|_{\hat x}$ is trivial, where $\hat{x}:=\eta\circ\tau^{-1}(x)$ for each $x\in U$.
\end{thm}
Define the {\it Sub-Laplacian} $$\Delta_{b}:=V_{00'}V_{11'}-V_{10'}V_{01'},$$
and {\it partial exterior differential operators} ${\rm d}_{0}$ and ${\rm d}_{1}$ by
\begin{equation}\label{d-1-2}
{\rm d}_{0}f:=V_{00'}f\cdot\theta^{00'}+V_{10'}f\cdot\theta^{10'},\qquad
{\rm d}_{1}f:=V_{01'}f\cdot\theta^{01'}+V_{11'}f\cdot\theta^{11'}
\end{equation}
for a function $f\in C^{\infty}(U,\mathbb{C})$.
By using   Atiyah-Ward ans\"{a}tz, we can construct a family of ASD connections.
\begin{thm}\label{ASDA}
(1) If $\varphi$ satisfies
\begin{equation}\label{eq:compatible-}
\Delta_{b}\varphi=0,
\end{equation}
then the connection form
 \begin{equation}\label{AASSDD11}
     \Phi
     =\left[\begin{array} {cc}
     \frac{1}{2}({\rm d}_{0}\ln{\varphi}-{\rm d}_{1}\ln{\varphi})
     & V_{01'}(\ln{\varphi})\theta^{00'}+V_{11'}(\ln\varphi)\theta^{10'}\\
     V_{00'}(\ln{\varphi})\theta^{01'}+V_{10'}(\ln\varphi)\theta^{11'}
     &\  -\frac{1}{2}({\rm d}_{0}\ln{\varphi}-{\rm d}_{1} \ln{\varphi})\end{array}\right]+\Phi_{T}\theta
  \end{equation}
is ASD.

(2) In particular,
\begin{equation}
\varphi:=\frac{1}{\left\|\mathbf{ y}\right\|^{4}-t^{2}},\qquad {\it where}\quad  \left\|\mathbf{ y}\right\|^{2}=\rm{det}\left[\begin{array}{lll}
y_{00'}  &y_{01'}\\
y_{10'} &y_{11'}
\end{array}\right],
\end{equation}
is the solution to (\ref{eq:compatible-}) on $\mathscr H\setminus\{\left\|\mathbf{ y}\right\|^{4}=t^{2}\}$.
\end{thm}
Now we consider $5$D real Heisenberg group $\mathscr H^{\mathbb{R}}\cong \mathbb{R}^{5}$ with multiplication given by
\begin{equation} \label{eq:w-multiplication-r} \begin{split}
&(\mathbf{y},s) \circ(\mathbf{y}', s' )
=\left ( \mathbf{y}+\mathbf{y}', s+s' +\langle\mathbf{ y} , \mathbf{y}'\rangle\right),
\end{split}\end{equation}
where $\langle\mathbf{ y} , \mathbf{y}'\rangle=2\left(y_{1}y_{2}'-y_{2}y_{1}'-y_{3}y_{4}'+y_{4}y_{3}'\right)$, $\mathbf{ y} , \mathbf{y}'\in \mathbb{R}^{4}$ and $s, s'\in \mathbb{R}$.
By the real imbedding $\mathbb{R}^{5}\longrightarrow\mathbb{C}^{5}$ given by
\begin{equation}\label{eq:C-L}
\left[\begin{array}{lll}
y_{00'}&y_{01'}\\
y_{10'}&y_{11'}
\end{array}\right]:=\left[\begin{array}{lll}
y_{1}+\textbf{i}y_{2}& -y_{3}+\textbf{i}y_{4}\\
y_{3}+\textbf{i}y_{4}& \ \ y_{1}-\textbf{i}y_{2}
 \end{array}\right],\quad t=-\textbf{i}s,
\end{equation}
$\mathscr H^{\mathbb{R}}$ is a subgroup of $\mathscr H$.
It is a flat model of 5D contact manifolds.

Recall the contact instanton equation on $5$D contact manifolds \cite{MR2967134} \cite{MR3042942}. For a connection $\nabla$, let us consider its Yang-Mills action $$YM(\nabla)=-\int_{\mathscr H^{\mathbb{R}}}Tr(F\wedge\ast F),$$
where $F$ is the curvature of $\nabla$ and $\ast$ is the Hodge star over $\mathscr H^{\mathbb{R}}$. $F_{H}$ and $F_{V}$ are horizontal and vertical part of $F$ respectively, which satisfy  $\iota_{T}F_{H}=0$ and $\iota_{T}F_{V}\neq0$ respectively.
 As $F_{H}\wedge\ast F_{V}=0$ and $F_{V}\wedge\ast F_{H}=0$, we have
$$YM(\nabla)=-\int_{\mathscr H^{\mathbb{R}}}Tr(F_{H}\wedge\ast F_{H}+F_{V}\wedge\ast F_{V}).$$
Take $F^{+}_{H}$ and $F^{-}_{H}$ as horizontal self-dual and   anti-self-dual parts of $F$, which satisfy
$\iota_{T}\ast F^{+}_{H}=F^{+}_{H}$ and $\iota_{T}\ast F^{-}_{H}=-F^{-}_{H}$ respectively.
  $F^{+}_{H}=0$ or $F^{-}_{H}=0$ together with $F_{V}=0$ are critical points of $YM(\nabla)$.
The anti-self-dual contact instanton equation  \cite{MR2967134} \cite{MR3042942} is
\begin{equation}\label{instan111}F^{+}_{H}=0\quad and \quad F_{V}=0,\end{equation}
while the self-dual one is
$$F^{-}_{H}=0\quad and \quad F_{V}=0.$$
The ASD equations (\ref{eq:ASD}) restricted to $\mathscr H^{\mathbb{R}}$ is exactly $F^{+}_{H}=0$. So it satisfies (\ref{instan111}) without $F_{V}=0.$
If $\varphi$ in (\ref{AASSDD11}) is replaced by
\begin{equation*}
\varphi^{\mathbb{R}}=\frac{1}{|y|^{4}+s^{2}},\quad where \ |y|=(y_{1}^{2}+y_{2}^{2}+y_{3}^{2}+y_{4}^{2})^{\frac{1}{2}},
\end{equation*}
we get ASD connection forms on $5$D real Heisenberg group $\mathscr H^{\mathbb{R}}$.

Baston and Easteood \cite{BE} generalize the twistor theory to a general setting based an the double fibration
\begin{equation}  \label{eq:G/P}      \xy 0;/r.22pc/:
 (0,20)*+{G/(P\cap Q) }="1";
(-20,0)*+{G/Q}="2";
(20,0)*+{G/P}="3";
{\ar@{->}_{\eta}"1"; "2"};
{\ar@{->}^{\tau} "1"; "3"};
\endxy,
 \end{equation}
where $G$ is a complex semisimple Lie group, $P$, $Q$ are its parabolic subgroups. If we take $G={\rm SO}(6,\mathbb{C})$ and suitable subgroups $P$ and $Q$, by using the method in \cite{Wa13}, we can also write down local coordinate charts of the above homogeneous spaces and the mapping $\eta$ and $\tau$ in terms of local coordinates to obtain (\ref{TOH}).

The construction of this paper is based on the fact that $V_{0}$ and $V_{1}$ in (\ref{keyvf}) span an abelian subalgebra. Its real version plays a very important role in developing a theory of quaternionic Monge-Amp\`ere operator in \cite{2020The} and tangential $k$-Cauchy-Fueter complex \cite{Ren1} over the Heisenberg group.

In Section 2, we introduce the twistor transform for $5$D complex Heisenberg group and derive the ASD Yang-Mills equation. In Section 3, we give Penrose-Ward correspondence between ASD connections and holomorphic vector bundles over the twistor space, which are trivial over a class of projective lines in the twistor space, and   construct a family of ASD connections by using Atiyah-Ward ans{\"a}tz. In Section 4, by  the real imbedding of $\mathscr H^{\mathbb{R}}$ into $\mathscr H$, we find that the ASD Yang-Mills equation coincides with the horizontal part of the contact instanton equation.
In Appendix, by constructing local coordinate charts of   homogeneous spaces in the double fibration (\ref{eq:G/P}) with  ${\rm G}={\rm SO}(6,\mathbb{C})$, we reproduce the basic ingredients of twistor method for  $5$-D complex Heisenberg group.

\section{The twistor transform on $5$D complex Heisenberg group}
\subsection{$\alpha$-planes and the twistor space of $5$D complex Heisenberg group}

Define a symmetric product $\langle\cdot, \cdot\rangle$ on $\mathbb{C}^{2}$ by
\begin{equation}\label{eq:beta}
\langle\mathbf{ w}  , \mathbf{\widetilde{w}}\rangle:=w_{1 } \widetilde{w}_{2}+\widetilde{w}_{1 }{w}_{2 }
\end{equation}
for $\mathbf{ w}=(w_{1},w_{2})$, $\mathbf{\widetilde{w}}=(\widetilde{w}_{1},\ \widetilde{w}_{2})\in \mathbb{C}^{2}$.
If we denote $\mathbf{y}=(\mathbf{y}_{0'},\mathbf{y}_{1'})$ with
$\mathbf{y}_{0'}=(y_{00'},y_{10'})$
and $\mathbf{y}_{1'}=(y_{01'},y_{11'})$, the multiplication (\ref{eq:w-multiplication}) of the Heisenberg group can be also written as
\begin{equation} \label{eq:w-multiplication2} \begin{split}&(\mathbf{y}_{0'}, \mathbf{ y}_{1'},t  ) \circ
(\mathbf{\widetilde{y}}_{0'}, \mathbf{\widetilde{y}}_{1'},\widetilde{t})=\left ( \mathbf{y}_{0'}+\mathbf{\widetilde{y}}_{0'}, \mathbf{ y}_{1'}+\mathbf{\widetilde{y}}_{1'},t+\widetilde{t} +
\langle\mathbf{ y}_{0'} , \mathbf{\widetilde{y}}_{1'}\rangle-\langle\mathbf{ y}_{1'} , \mathbf{\widetilde{y}}_{0'}\rangle\right).
\end{split}\end{equation}
 Recall the {\it
   left translation}: for fixed $(\mathbf{y}',t')\in\mathscr{H},$
\begin{equation}\label{l-t}
\tau_{(\mathbf{y}',t')}:
(\mathbf{y}, t)\mapsto(\mathbf{y}',t')\cdot(\mathbf{y},t) ,\qquad\qquad (\mathbf{y},t)\in\mathscr{H}
\end{equation}
and  the {\it   dilation}:
\begin{equation}
\delta_{r}:(\mathbf{y},t)\mapsto(r\mathbf{y},r^{2}t),
\end{equation}on the Heisenberg group.
A vector field $V$ over $\mathscr{H}$ is called {\it left invariant} if for any $(\mathbf{y}',t')\in\mathscr{H}$, we have
 $$\tau_{(\mathbf{y}',t')*}V=V,$$ where $\tau_{(\mathbf{y}',t')}$ is the left translation in (\ref{l-t}).
Define
\begin{equation} \label{eq:left-invariant}
(V_{AA'}f)(\mathbf{y} ,t ):=\left.\frac{\hbox{d}}{\hbox{d}s}f((\mathbf{y} ,t )(se_{AA'}))
\right|_{s=0},\qquad (Tf)(\mathbf{y} ,t ):=\left.\frac{\hbox{d}}{\hbox{d}s}f\left((\mathbf{y},t )(se_{0})\right)
\right|_{s=0}
 \end{equation}
for $A=0,1,A'=0',1',$ where $e_{AA'}$ is a vector in $\mathbb{C}^{5}$ with all entries vanishing except for  the $(AA')$-entry   to be $1$, and
$e_{0}=(0,0,0,0,1)$. For example,
\begin{equation} \label{eq:left-invariant2}\begin{split}
V_{00'}f:=\left.\frac{\hbox{d}}{\hbox{d}s}f\left((\mathbf{y} ,t )(se_{00'})\right)
\right|_{s=0}&=\left.\frac{\hbox{d}}{\hbox{d}s}f\left((\mathbf{y} ,t )(s,0,0,0,0)\right)
\right|_{s=0}\\
&=\left.\frac{\hbox{d}}{\hbox{d}s}f\left(y_{00'}+s,y_{10'},y_{01'},y_{11'},t-sy_{11'})\right)
\right|_{s=0}\\
&=\left(\frac{\partial}{\partial{y_{00'}}}-y_{11'}\frac{\partial}{\partial t}\right)f.
\end{split} \end{equation}

We can describe $\alpha$-planes, the integral surfaces of $V_{0}$ and $V_{1}$, explicitly as follows.
$\mathbb{C}^{5}\times\mathbb{C}P^{1}$ is the complex manifold with two coordinate charts $\mathbb{C}^{5}\times\mathbb{C}$ and $\mathbb{C}^{5}\times\mathbb{C}$, glued by the mapping $\kappa: \mathbb{C}^{5}\times\mathbb{C}\setminus \{0\}\rightarrow\mathbb{C}^{5}\times\mathbb{C}\setminus \{0\}$ given by \begin{equation}\label{trantran1}(\mathbf{y},\zeta)\longmapsto(\mathbf{y},\zeta^{-1}).
\end{equation}
Then if we use the nonhomogeneous coordinates,
$\tau:\mathbb{C}^{5}\times\mathbb{C}P^{1}\longrightarrow\mathbb{C}^{5}$ is given by  $(\mathbf{y},\zeta)\longrightarrow\mathbf{y} $ and  $(\mathbf{y},\widetilde{\zeta})\longrightarrow\mathbf{y} $, and
the vector field $V_{A}$ in (\ref{keyvf}) can be rewritten as
$$V_{A}=\pi_{1'}V_{A}^{\zeta}=\pi_{0'}\widetilde{V}_{A}^{\widetilde{\zeta}},$$
where $\zeta=\frac{\pi_{0'}}{\pi_{1'}}$, $\widetilde{\zeta}=\frac{\pi_{1'}}{\pi_{0'}}$ and $$V_{A}^{\zeta}=\zeta V_{A0'}-V_{A1'},\qquad\widetilde{V}_{A}^{\widetilde{\zeta}}= V_{A0'}-\widetilde{\zeta}V_{A1'}.$$
Let us check that the integral  surfaces  of $V_{A}^{\zeta}$ and $V_{A}^{\widetilde{\zeta}}$ lifted to $\mathbb{C}^{5}\times\mathbb{C}P^{1}$ by $\tau$ are the fiber of the mapping $\eta:\mathbb{C}^{5}\times\mathbb{C}\rightarrow\mathbb{C}^{4}$ and $\widetilde{\eta}:\mathbb{C}^{5}\times\mathbb{C}\rightarrow\mathbb{C}^{4}$, respectively.
\begin{prop}
Let $\eta:\mathbb{C}^{5}\times \mathbb{C}\rightarrow W\cong\mathbb{C}^{4}$ be the mapping given by
\begin{equation}\label{eq:psi}
\mathbf{\omega}=\eta(\mathbf{y},t,\zeta)
=
\begin{pmatrix}
\eta_{0}(\mathbf{y},t,\zeta)\\
\eta_{1}(\mathbf{y},t,\zeta)\\
\eta_{2}(\mathbf{y},t,\zeta)\\
\eta_{3}(\mathbf{y},t,\zeta)
\end{pmatrix}
=
\begin{pmatrix}
y_{00'}+\zeta y_{01'}\\
y_{10'}+\zeta y_{11'}\\
t-\langle \mathbf{y}_{0'}+\zeta\mathbf{y}_{1'},\mathbf{y}_{1'}\rangle\\
\zeta
\end{pmatrix}\in W.
\end{equation}
Then $\tau\circ\eta^{-1}(w)$ is a $2$-D plane parameterized as
\begin{equation}\label{eq:lienar eq0}
     \left\{\begin{array}{l}
y_{01'}=s_{0},\\
y_{11'}=s_{1},\\
y_{00'}=\omega_{0}-\zeta s_{0},\\
y_{10'}=\omega_{1}-\zeta s_{1},\\
t=\omega_{2}+s_{1}\omega_{0}+s_{0}\omega_{1},
         \end{array}  \right.
  \end{equation}
with parameters $s_{0},s_{1}\in \mathbb{C}$. $V_{0}^{\zeta}$ and $V_{1}^{\zeta}$ are tangential to this plane, and so it is an $\alpha$-plane.
\end{prop}
\begin{proof}
It is direct to see that $V_{AA'}(y_{BB'})=\delta_{AB}\delta_{A'B'}$ by the expression of $V_{AA'}$'s in (\ref{eq:Y}). So we have
 $$V_{A}^{\zeta}(\eta_{j}(\mathbf{y},t,\zeta))=0,\quad j=0,1.$$
Noting that
\begin{equation}\label{currr000}\langle \mathbf{y}_{0'}+\zeta\mathbf{y}_{1'},\mathbf{y}_{1'}\rangle=y_{00'}y_{11'}+y_{10'}y_{01'}+2\zeta y_{01'}y_{11'},\end{equation}
 we have
\begin{equation*}\begin{split}
V_{0}^{\zeta}(\eta_{2}(\mathbf{y},t,\zeta))&
=\zeta\left(-y_{11'}-y_{11'}\right)-\left(y_{10'}-y_{10'}-2\zeta y_{11'}\right)=0,\\
V_{1}^{\zeta}(\eta_{2}(\mathbf{y},t,\zeta))&
=\zeta\left(-y_{01'}-y_{01'}\right)-\left(y_{00'}-y_{00'}-2\zeta y_{01'}\right)=0.
\end{split}\end{equation*}
Thus $V_{0}^{\zeta}$ and $V_{1}^{\zeta}$ are tangential to each fiber of $\eta$.
Note that for a fixed point $\mathbf{\omega}=(\omega_{0},\omega_{1},\omega_{2},\zeta)$ $\in W$, $\eta^{-1}(w)$ in $\mathbb{C}^{5}\times\mathbb{C}$ has fixed last coordinate $\zeta$. So $\tau\circ\eta^{-1}(w)$ is the plane determined by
\begin{equation}\label{newlss}
\eta_{0}(\mathbf{y},t,\zeta)=\omega_{0},\quad\eta_{1}(\mathbf{y},t,\zeta)=\omega_{1},
\quad\eta_{2}(\mathbf{y},t,\zeta)=\omega_{2}.
\end{equation}
The solutions of linear equations
$
     \left\{\begin{array}{l}
y_{00'}+\zeta y_{01'}=\omega_{0}\\
y_{10'}+\zeta y_{11'}=\omega_{1}
         \end{array}   \right.
$
are given by
$y_{01'}=s_{0},\ y_{00'}=\omega_{0}-s_{0}\zeta,\ y_{11'}=s_{1},\ y_{10'}=\omega_{1}-s_{1}\zeta$ with parameters $s_{0},s_{1}\in \mathbb{C}$.
Then $t$ is given by the last equation in (\ref{eq:lienar eq0}) by the third equation of (\ref{eq:psi}).
\end{proof}

On the other hand, if $\pi_{0'}\neq0,$   integral surfaces of $\widetilde{V}_{0}^{\widetilde{\zeta}}$ and $\widetilde{V}_{1}^{\widetilde{\zeta}}$ are fibers of the mapping $\widetilde{\eta}:\mathbb{C}^{5}\times\mathbb{C}\longrightarrow \widetilde{W}\cong\mathbb{C}^{4}$ given by
\begin{equation}\label{eq:psi2}
\widetilde{\mathbf{\omega}}=\widetilde{\eta}(\mathbf{y},t,\widetilde{\zeta})
=
\begin{pmatrix}
\widetilde{\eta}_{0}(\mathbf{y},t,\widetilde{\zeta})\\
\widetilde{\eta}_{1}(\mathbf{y},t,\widetilde{\zeta})\\
\widetilde{\eta}_{2}(\mathbf{y},t,\widetilde{\zeta})\\
\widetilde{\eta}_{3}(\mathbf{y},t,\widetilde{\zeta})
\end{pmatrix}
=
\begin{pmatrix}
\widetilde{\zeta}y_{00'}+y_{01'}\\
\widetilde{\zeta}y_{10'}+y_{11'}\\
t+\langle\ \widetilde{\zeta}\mathbf{y}_{0'}+\mathbf{y}_{1'},\mathbf{y}_{0'}\rangle\\
\widetilde{\zeta}
\end{pmatrix}\in\widetilde{W},
\end{equation}
and for $\widetilde{\mathbf{\omega}}=(\widetilde{w}_{0},\widetilde{w}_{1},\widetilde{w}_{2},\widetilde{\zeta})\in\widetilde{W}$, the $\alpha$-plane $\tau\circ\widetilde{\eta}^{-1}(\widetilde{\mathbf{\omega}})$ is given by
\begin{equation}\label{eq:lienar eq20}
     \left\{\begin{array}{l}
y_{01'}=\widetilde{s}_{0}\\
y_{11'}=\widetilde{s}_{1}\\
y_{00'}=\widetilde{\omega}_{0}-\widetilde{\zeta}\widetilde{s}_{0}\\
y_{10'}=\widetilde{\omega}_{1}-\widetilde{\zeta}\widetilde{s}_{1}\\
t=\widetilde{\omega}_{2}+\widetilde{s}_{1}\widetilde{\omega}_{0}+\widetilde{s}_{0}\widetilde{\omega}_{1},
         \end{array}   \right.
  \end{equation}
with parameters $\widetilde{s}_{0},\widetilde{s}_{1}\in \mathbb{C}$.
If $(\mathbf{y},t,\zeta)$ satisfies (\ref{newlss}) with $\omega_{0}$, $\omega_{1}$, $\omega_{2}$, $\zeta=\widetilde{\zeta}^{-1}\in\mathbb{C}$, then we have
\begin{equation}
\widetilde{\eta}_{A}(\mathbf{y},t,\widetilde{\zeta})=\zeta^{-1}\omega_{A},\quad A=0,1,
  \end{equation}
  and
\begin{equation}\label{lastt11}
\begin{split}
\widetilde{\eta}_{2}(\mathbf{y},t,\widetilde{\zeta})&=t+\langle\ \zeta^{-1}\mathbf{y}_{0'}+\mathbf{y}_{1'},\mathbf{y}_{0'}\rangle\\
&=t+y_{01'}y_{10'}+y_{11'}y_{00'}+2\zeta^{-1}y_{00'}y_{10'}\\
&=t-\langle \mathbf{y}_{0'}+\zeta\mathbf{y}_{1'},\mathbf{y}_{1'}\rangle+2\zeta^{-1}(y_{00'}+\zeta y_{01'})(y_{10'}+\zeta y_{11'})\\
&=\omega_{2}+2\zeta^{-1}\omega_{0}\omega_{1},
         \end{split}
  \end{equation}
by (\ref{currr000}).
So $\kappa(\mathbf{y},t,\zeta)=(\mathbf{y},t,\widetilde{\zeta})$ maps a fiber of $\eta$ over $(\omega_{0},\omega_{1},\omega_{2},\zeta)$ to a fiber of $\widetilde{\eta}$ over
 $(\widetilde{\omega}_{0},\widetilde{\omega}_{1},\widetilde{\omega}_{2},\widetilde{\zeta})$
with the mapping
 \begin{equation}\label{transi1}\begin{split}
 \Phi:\quad W\setminus\{\zeta=0\}&\rightarrow \widetilde{W}\setminus\{\widetilde{\zeta}=0\}\\
 (w_0,w_1,w_2,\zeta)&\mapsto(\widetilde{\omega}_{0},\widetilde{\omega}_{1},\widetilde{\omega}_{2},\widetilde{\zeta})=(\zeta^{-1}w_0, \zeta^{-1}w_1,w_2+2\zeta^{-1}w_{1}w_{2} ,\zeta^{-1}),
 \end{split}\end{equation}
which glues $W$ and $\widetilde{W}$ to get a complex manifold $\mathcal{P}$. It is the moduli space of all $\alpha$-planes,
which is our twistor space.

Moreover, we have the commutative diagram
\begin{equation}\label{dia11}\xymatrix{
  \mathbb{C}^{5}\times\mathbb{C}\ar[r]^-{\kappa}\ar[d]^{\eta}&
  \mathbb{C}^{5}\times\mathbb{C}\ar[d]^{\widetilde{\eta}}&
  \\
  W\ar[r]_{\Phi}&\widetilde{W}&
   }
  \end{equation}
In fact, for $\left(\mathbf{y}_{0'},\mathbf{y}_{1'},\zeta\right)\in\mathbb{C}^{5}\times\mathbb{C}$,
\begin{equation}\begin{split}
\Phi\circ\eta\left(\mathbf{y}_{0'},\mathbf{y}_{1'},\zeta\right)
&=\Phi\left(\mathbf{y}_{0'}+\zeta\mathbf{y}_{1'},t-\langle \mathbf{y}_{0'}+\zeta\mathbf{y}_{1'},\mathbf{y}_{1'}\rangle,\zeta\right)\\
&=\left(\zeta^{-1}\mathbf{y}_{0'}+\mathbf{y}_{1'},t+\langle \zeta^{-1}\mathbf{y}_{0'}+\mathbf{y}_{1'},\mathbf{y}_{0'}\rangle,\zeta^{-1}\right)\\
&=\widetilde{\eta}\circ\kappa\left(\mathbf{y}_{0'},\mathbf{y}_{1'},\zeta\right),
\end{split}\end{equation}
by (\ref{lastt11}). Thus $\eta$ and $\widetilde{\eta}$ are glued to give the mapping $\mathbb{C}^{5}\times \mathbb{C}P^{1}\longrightarrow\mathcal{P}$ in the double fibration (\ref{TOH}).

\subsection{ASD Equation}
Let $1$-forms $\{\theta^{AA'},\theta\}$ be dual to left invariant vector fields $\{V_{AA'},T\}$ in (\ref{eq:Y}) on $\mathscr H$, i.e.
$
\theta^{AA'}(V_{BB'})=\delta_{AB}\delta_{A'B'},\ \theta^{AA'}(T)=0,\
   \theta(V_{BB'})=0, \ \theta(T)=1,
$
where $A,B=0,1$ and $A',B'=0',1'$.
Then ${\rm{d}}u= \sum_{A,A'} V_{AA'}u\cdot\theta^{AA'}+Tu\cdot\theta$ for a function $u$ on $\mathscr H$.
By the expression of $V_{AA'}$ in (\ref{eq:Y}), we get that
$
\theta^{AA'}={\rm{d}}y_{AA'},\
\theta={\rm{d}}t+y_{11'}{\rm{d}}y_{00'}+y_{01'}{\rm{d}}y_{10'}-y_{10'}{\rm{d}}y_{01'}-y_{00'}{\rm{d}}y_{11'}.
$
Exterior differentiation gives us
${\rm{d}}\theta^{AA'}=0$ $(A=0,1, A'=0',1')$
and ${\rm{d}}\theta=-2\theta^{00'}\wedge \theta^{11'}-2\theta^{10'}\wedge \theta^{01'}.$

The curvature of the connection form $\Phi$ is $F =({\rm{d}}+\Phi)^2={\rm{d}}\Phi+\Phi\wedge \Phi$ given by
  \begin{equation}\label{eq:curvature}\begin{split}
     F(X,Y)&={\rm{d}}\Phi(X,Y)+\Phi\wedge \Phi(X,Y)\\&=X(\Phi(Y))-Y(\Phi(X))- \Phi([X,Y]) +\Phi(X)\Phi(Y)-\Phi(Y)\Phi(X)
     \\&=X\Phi_Y-Y \Phi_X- \Phi_{[X,Y]} +[\Phi_X,\Phi_Y].
    \end{split}  \end{equation}Here we use the notation $\Phi_X:=\Phi(X)$ for a vector field $X$ on $\mathscr H$.
Define the $\mathfrak g$-valued differential operators associated to the connection form $\Phi$
\begin{equation}\label{eq:Dj}\begin{split}&
   \nabla_A =\nabla_{A1'}-\zeta \nabla_{A0'}:=(V_{A1'}+\Phi_{A1'})-\zeta(V_{A0'}+\Phi_{A0'}),
    \end{split}\end{equation}
$A=0,1,$ for fixed $\zeta\in\mathbb{C}$. A connection on $\mathscr{H}$ is {\it ASD } if its curvature vanishes along each $\alpha$-surface, i.e.
\begin{equation}\label{eq:yang}
F(V_0^{\zeta},V_1^{\zeta})= 0.
\end{equation}
The ASD condition (\ref{eq:yang}) is equivalent to
\begin{equation*}
\zeta^{2}F(V_{00'},V_{10'})-\zeta(F(V_{00'},V_{11'})+F(V_{01'},V_{10'}))+F(V_{01'},V_{11'})=0.
\end{equation*}
Comparing the coefficients of $\zeta^{2}$, $\zeta^{1}$ and $\zeta^{0}$, we get
\begin{equation}\label{eq:Y5}F( V_{00'},V_{10'})=0,\quad F( V_{00'},V_{11'})+F( V_{01'},V_{10'})=0,\quad F( V_{01'},V_{11'})=0,\end{equation}
which is equivalent to (\ref{eq:ASD}) by (\ref{eq:curvature}). Here $\Phi_{[V_{00'},V_{11'}]}+\Phi_{[V_{01'},V_{10'}]}=0$ by the brackets in (\ref{eq:YYT}).

\section{Penrose-Ward correspondence and Atiyah-Ward ans\"{a}tz on   $5$D Heisenberg group}

\subsection{ Proof of Theorem \ref{pwc}}
The proof is similar to the classical case (cf. \cite{MR1054377}). The objects in (1) and (2) are regarded as being specified modulo the usual equivalence relations.
Given a $GL(n,\mathbb{C})$ vector bundle $V$ with ASD connection $\nabla$ over $U\subseteq\mathbb{C}^{5}$, to construct a vector bundle $E$ over $\hat{U}$, we assign a copy of the vector space $\mathbb{C}^{n}$ to each point $Z$ of $\hat{U}$ with
 \begin{equation}\label{NBA}
E_{Z}=\{\psi:\nabla\psi|_{\widetilde{Z}\bigcap U}=0\}
\end{equation}
over $Z$, where $\widetilde{Z}=\tau\circ\eta^{-1}(Z).$ As $\nabla$ is an ASD connection,
for fixed $\zeta\in\mathbb{C}$, the integrability condition $F(V_0,V_1)=0$ implies the existence of solutions to the equations
$(V_{A}+\Phi_{A})h=0,$ $A=0,1$ on connected and simply connected domain $\widetilde{Z}\bigcap U$.
So we have $E_{Z}\neq\emptyset$. Since the whole procedure is holomorphic, we have constructed a holomorphic vector bundle $E$ over $\hat{U}$.
By construction (\ref{NBA}), the bundle $E$ is trivial when restricted to the projective line $\hat{x}$ for a point $x$ in $U$. This is because a vector $\psi\in V_{x}$ at $x$ determines a parallel field $\psi$ on each $\alpha$-plane through $x$, and hence determines a point in $E_{Z}$ for every point $Z$ on the line $\hat{x}$ in $\hat{U}$. Namely, each $\psi\in V_{x}\cong\mathbb{C}^{n}$ determines a holomorphic section of $E|_{\hat{x}}$. Therefore $n$ linearly independent $\psi$'s in $V_{x}$ give us $n$ linearly independent sections of $E|_{\hat{x}}$. So $E|_{\hat{x}}$ is trivial. This completes the proof of the first part of the theorem.

Conversely, let $E'$ be a holomorphic rank-$n$ vector bundle over $\hat{U}$, such that $E'|_{\hat{x}}$ is trivial for all $x\in U$. We have to construct a connection form $\Phi$ on $U$. Since $\hat{U}$ is covered by two charts $\hat{U}\cap W$ and $\hat{U}\cap\widetilde{W}$ in (\ref{transi1}), the vector bundle $E'$ consists of two parts, which are $E'_{\hat{U}\cap W}\cong (\hat{U}\cap W)\times\mathbb{C}^{n}$ and $E'_{\hat{U}\cap\widetilde{W}}\cong (\hat{U}\cap\widetilde{W})\times\mathbb{C}^{n}$ glued by a holomorphic $n\times n$ transition matrix $F$ on the intersection $\hat{U}\cap W\cap \widetilde{W}$.
The transition relation is
$$\widetilde{\xi}=F\xi,$$
 where $\widetilde{\xi}$ and $\xi$ are column $n$-vectors whose components serve as coordinates on the fibers of $E'$ above $\hat{U}\cap W$ and $\hat{U}\cap\widetilde{W}$, respectively.
 Consider the pull-back $\eta^{*}E'$, which is a bundle over $\mathcal{F}_{U}=U\times\mathbb{C}P^{1}$, where $\eta$ is given by (\ref{eq:psi}) and  (\ref{eq:psi2}). For a point $Z\in\hat{U}$, the restriction of $\eta^{*}E'$ to $\eta^{-1}(Z)\in\mathcal{F}_{U}$ is a product bundle.
 We define a bundle $E\longrightarrow U$ by $E_{x}=\Gamma(\hat{x},E')$, where $\Gamma$ denotes the space of holomorphic sections. And we have $\eta^{*}E'=\tau^{*}E.$
Recall that the tangential space of leaves of the projection $\eta :\mathcal{F}_{U}\longrightarrow\hat{U}$ are spanned by the vector fields $V_{0}$ and $V_{1}$ on $\mathcal{F}$. We define a partial connection $\nabla$ that allows us to differentiate the sections of $\eta^{*}E'$ along the fibers, where we have $\nabla_{V_{A}}s=V_{A}s$ in the trivialization
$\eta^{*}E'|_{\eta^{-1}(Z)}=\eta^{-1}(Z)\times E'_{Z}$. The sections for which $\nabla_{V_{A}}s=V_{A}s$ vanish are the pull-backs to $\mathcal{F}$ of local sections of $E'$.
We now pick a local trivialization of $E$ over open subset $U$. This determines a local trivializaton of $\eta^{*}E'$, in which
 $$\nabla_{A}=V_{A}+\Phi_{A},\quad A=0,1,$$
for some matrix-valued functions $\Phi$ in $\zeta$ and $(\mathbf{y},t)$.
Note that
$\eta^{*}E'|_{U\times O_0}\cong U\times O_0\times\mathbb{C}^{n}$ and $\eta^{*}E'|_{U\times O_1}\cong U\times O_1\times\mathbb{C}^{n}$, where $O_0\cong\mathbb{C}$ and $O_1\cong\mathbb{C}$ are a covering of $\mathbb{CP}^1$.
Their transition function $G$ over the intersection $U\times\left(\mathbb{C}\setminus\{0\}\right)$ is given by
\begin{equation*}
G(\mathbf{y},t,\zeta)=F\left(\mathbf{y}_{0'}+\zeta\mathbf{y}_{1'},t-\langle \mathbf{y}_{0'}+\zeta \mathbf{y}_{1'},\mathbf{y}_{1'}\rangle,\zeta\right).
\end{equation*}

  Now let us find the holomorphic sections of $E'|_{\hat{x}}$. Find non singular $n\times n$ matrices $f$ and $\widetilde{f}$, with $f$ holomorphic over $W\cap\hat{x}$ and $\widetilde{f}$ holomorphic over $\widetilde{W}\cap{\hat{x}}$, such that $F=\widetilde{f}^{-1}f$ is valid on $W\cap \widetilde{W}\cap \hat{x}$. Since $E'_{\hat{x}}$ is trivial, by the {\it Birkhoff factorization}, such matrices $f$ and $\widetilde{f}$ must exist. Each section of $E'_{\hat{x}}$ is then given by $\xi=f^{-1}\psi$, $\widetilde{\xi}=\widetilde{f}^{-1}\psi$, where $\psi$ is a constant $n$-vector, i.e. $\psi\in\mathbb{C}^{n}.$ Now we identify $\widetilde{f}$ and $f$ with their pulling back to $U\times\mathbb{CP}^1$.
So we have that
\begin{equation}\label{kpw}\begin{split}
0&=\nabla_{A}\psi=(-V_{A1'}+\zeta V_{A0'})\psi+(-\Phi_{A1'}+\zeta \Phi_{A0'})\psi\\
&=(-V_{A1'}f+\zeta V_{A0'}f)\xi+(-\Phi_{A1'}+\zeta \Phi_{A0'})\psi\\
&=-(V_{A1'}f\cdot f^{-1}+\Phi_{A1'})\psi+\zeta(V_{A0'}f\cdot f^{-1}+\Phi_{A0'})\psi.
\end{split}
\end{equation}
So we have $\Phi_{AA'}=-V_{AA'}f\cdot f^{-1}.$
Moreover, since $V_{A}$ is tangential to the fiber of $\eta$, we have
\begin{equation}\label{kpw2}\begin{split}
0&=V_{A}G
=V_{A}(\widetilde{f}^{-1}f)=-\widetilde{f}^{-1}V_{A}\widetilde{f}\cdot\widetilde{f}^{-1}f+\widetilde{f}^{-1}V_{A}f\\
&=\widetilde{f}^{-1}(V_{A}f\cdot f^{-1}-V_{A}\widetilde{f}\cdot\widetilde{f}^{-1})f.
\end{split}
\end{equation}
We have $V_{A}f\cdot f^{-1}=V_{A}\widetilde{f}\cdot\widetilde{f}^{-1}$.
By Liouville's theorem, both sides must be of the form $-\Phi_{A1'}+ \zeta\Phi_{A0'}$ for $A=0,1$, where $\Phi_{A0'}$ and $\Phi_{A1'}$ are matrix-valued functions over $\mathbb{C}^{5}$. Therefore
\begin{equation*}
V_{A1'}f\cdot f^{-1}-\zeta V_{A0'}f\cdot f^{-1}=-\Phi_{A1'}+ \zeta\Phi_{A0'},
\end{equation*}
i.e.
 \begin{equation}\label{eq:asd-eq}\begin{split}&
  ( V_{A1'}+\Phi_{A1'})f-\zeta ( V_{A0'}+\Phi_{A0'})f=0.
\end{split}\end{equation}
Thus $f$ is the solution to
\begin{equation*}
   \nabla_{A1'}f-\zeta \nabla_{A0'}f=0,
\end{equation*}
 $A=0,1$. Consequently, we have $F(V_{0},V_{1})=0,$ i.e. $\nabla$ is ASD connection.
Similarly, $\widetilde{f}$ is solution to
\begin{equation}\label{eq:asd-eq2}\begin{split}&
  ( V_{A0'}+\Phi_{A0'})\widetilde{f}-\widetilde{\zeta}( V_{A1'}+\Phi_{A1'})\widetilde{f}=0,
\end{split}\end{equation}
i.e.
$\nabla_{A0'}\widetilde{f}-\widetilde{\zeta}\nabla_{A1'}\widetilde{f}=0,
$ where $A=0,1$.
From (\ref{eq:asd-eq}) and (\ref{eq:asd-eq2}), we have
\begin{equation}\label{eq:asd-eq3}\begin{split}
&\Phi_{A1'}=-V_{A1'}f\cdot f^{-1}|_{\zeta=0},\qquad
\Phi_{A0'}=-V_{A0'}\widetilde{f}\cdot\widetilde{f}^{-1}|_{\widetilde{\zeta}=0},
\end{split}\end{equation}
where $A=0,1$. So the connection $\Phi$ has the form
\begin{equation}\label{eq:asd-eq4}
\Phi=-{\rm d}_{1}f\cdot f^{-1}|_{\zeta=0}-{\rm d}_{0}\widetilde{f}\cdot \widetilde{f}^{-1}|_{\widetilde{\zeta}=0}+\Phi_{T}\theta,
\end{equation}
which is ASD.
\qed

\subsection{Atiyah-Ward ans\"{a}tz For ${\rm GL}(2,\mathbb{C})$  ASD connection on 5D Heisenberg group}
By the construction in the proof of Theorem \ref{pwc}, we get a ${\rm GL}(2,\mathbb{C})$ ASD connection if we
 find a transition matrix for a rank-2 bundle over $\hat{U}$, which is trivial along the projective line
  $\hat{x}=\eta\circ\tau^{-1}(x)$ for each $x\in U\subseteq\mathscr H$. Namely, to find a $2\times2$ matrix
$$G(\mathbf{y},t,\zeta)=\left(\begin{array} {cc} \zeta&\gamma(\mathbf{y},t,\zeta)\\0&\zeta^{-1} \end{array}\right)$$
 defined over $U\times \left(\mathbb{C}\setminus\{0\}\right)$ (the intersection of two coordinate charts of $\mathbb{C}^{5}\times \mathbb{C}P^{1}$), and is constant along the fibers of $\eta$ and trivial over $\hat{\mathbf{x}}$ for each $\mathbf{x}=(\mathbf{y},t)\in U\subseteq\mathbb{C}^{5}$.
It defines a function $\hat{G}$ on $\hat{U}\cap W\cap\widetilde{W}$
$$G(\mathbf{y},t,\zeta)=  \hat{G}\left(\mathbf{y}_{0'}+\zeta \mathbf{y}_{1'},t-\langle \mathbf{y}_{0'}+\zeta \mathbf{y}_{1'},\mathbf{y}_{1'}\rangle,\zeta\right).$$
This is called {\it Atiyah-Ward Ans\"{a}tz}.

In terms of Laurent series of $\zeta$, we write
\begin{equation}\label{cure112233000}
\gamma=\sum_{i=-\infty}^{+\infty} \gamma_{-i}\zeta^{i}=\gamma_{-}+\gamma_{0}+\gamma_{+},
\end{equation}
where $\gamma_{-}:=\sum_{i=1}^{i=\infty} \gamma_{i}\zeta^{-i}$, $\gamma_{+}:=\sum_{i=1}^{i=\infty} \gamma_{-i}\zeta^{i}$.
Since $\gamma$ need to be constant along each fiber of $\eta$, we have $V_{A}\gamma\equiv0$, $A=0,1,$ i.e.
\begin{equation}\label{eq:induct}
V_{A1'}\gamma_{i}=V_{A0'}\gamma_{i+1},\quad A=0,1,\quad i=\cdots,-1,0,1,\cdots.
\end{equation}
\begin{lem}\label{Poin}
Suppose that $U\subseteq\mathbb{C}^{5}$ is elementary and the one form $\omega=g_{0}\theta^{00'}+g_{1}\theta^{10'}$ (or $\omega=g_{0}\theta^{01'}+g_{1}\theta^{11'}$)
with $g_{0},g_{1}\in \mathcal{O}(U)$ is $\hbox{d}_{0}$-closed (or $\hbox{d}_{1}$-closed),
i.e. $\hbox{d}_{0}\omega=0$ (or $\hbox{d}_{1}\omega=0$).
Then there exists a function $f\in \mathcal{O}(U)$, such that
$\hbox{d}_{0}f=\omega$ (or $\hbox{d}_{1}f=\omega$).
\end{lem}
\begin{proof}
 By definition, $\hbox{d}_{0}\omega=0$ is equivalent to
\begin{equation}\label{eq:Y3}
V_{10'}g_{0}=V_{00'}g_{1},
\end{equation}
and $\hbox{d}_{0}f=\omega$ is equivalent to
\begin{equation}\label{eq:Y4}
         V_{00'}f=g_{0}\quad and\quad
         V_{10'}f=g_{1}.
\end{equation}
If we take coordinate transformation $\Psi:\mathbb{C}^{5}\longrightarrow\mathbb{C}^{5}$ given by
$$(y_{00'},y_{10'},y_{01'},y_{11'},t):=\Psi(z_{0},z_{1},z_{2},z_{3},z_{4})=\left(z_{0},z_{1},z_{2},z_{3},z_{4}-z_{0}z_{3}
-z_{1}z_{2}\right)
,$$
we have
$$\Psi_{*}\frac{\partial}{\partial{z_{0}}}=V_{00'},\qquad \Psi_{*}\frac{\partial}{\partial{z_{1}}}=V_{10'},$$
by expression of $V_{AA'}$ in (\ref{eq:Y}). Take $G_{A}=g_{A}\circ\Psi$, $A=0,1,$ and $F=f\circ \Psi$.
Then by pulling back, we need to solve
\begin{equation}\label{eq:poincare}
 \frac{\partial F}{\partial z_{0}}=G_{0},
\quad \frac{\partial F}{\partial z_{1}}=G_{1}.
\end{equation}
Under the condition $\frac{\partial G_{0}}{\partial z_{1}}=\frac{\partial G_{1}}{\partial z_{0}}$, it follows from Poincar\'e's lemma that (\ref{eq:poincare}) has a solution.
Therefore, $f=F\circ\Psi^{-1}$ is the solution to (\ref{eq:Y4}).
\end{proof}

\begin{lem}
Suppose $\gamma_{0}$ satisfy the equation (\ref{eq:compatible-}), i.e. $\Delta_{b}\gamma_{0}=0,$ then (\ref{eq:induct}) is solvable.
\end{lem}

\begin{proof}
Inductively, for fixed $i=0,-1,\cdots$, assuming that there exists $\gamma_{i}$ and $\gamma_{i-1}$ satisfying
\begin{equation}\label{eq:orig}
     \left\{\begin{array}{l}
         V_{01'}\gamma_{i-1}=V_{00'}\gamma_{i},\\
         V_{11'}\gamma_{i-1}=V_{10'}\gamma_{i},
         \end{array}  \right.
  \end{equation}
we need to find $\gamma_{i-2}$ such that
 \begin{equation}\label{eq:Ward}
     \left\{\begin{array}{l}
         V_{01'}\gamma_{i-2}=V_{00'}\gamma_{i-1},\\
         V_{11'}\gamma_{i-2}=V_{10'}\gamma_{i-1}.
         \end{array}  \right.
  \end{equation}
 Denote
  \begin{equation*}
  \Lambda_{i-1}:=V_{00'}\gamma_{i-1}\theta^{01'}+V_{10'}\gamma_{i-1}\theta^{11'}.  \end{equation*}
Then $\Lambda_{i-1}$ is ${\rm d}_{1}$-closed, since
\begin{equation*}
\begin{split}
{\rm d}_{1}\Lambda_{i-1}&
=(V_{01'}V_{10'}-V_{11'}V_{00'})\gamma_{i-1}\theta^{01'}\wedge\theta^{11'}
=(V_{10'}V_{01'}-V_{00'}V_{11'})\gamma_{i-1}\theta^{01'}\wedge\theta^{11'}\\&
=(V_{10'}V_{00'}-V_{00'}V_{10'})\gamma_{i}\theta^{01'}\wedge\theta^{11'}
=0,
\end{split}
 \end{equation*}
by using $[V_{00'},V_{11'} ]=[V_{10'},V_{01'} ]=T$ in (\ref{eq:YYT}) in the second identity, (\ref{eq:orig}) and $[V_{00'},V_{10'} ]=0$ in the last identity.
It follows from Lemma \ref{Poin} that there exists $\gamma_{i-2}$ such that $d_{1}\gamma_{i-2}=\Lambda_{i-1}$, i.e. equation (\ref{eq:Ward}) is satisfied.
So solving the equations (\ref{eq:induct}) for $i=\cdots,-1,0$, is reduced to the solve the equation $d_{1}\Lambda_{0}=0$ by induction. But
\begin{equation*}\begin{split}
{\rm d}_{1}\Lambda_{0}=&\left(V_{01'}V_{10'}-V_{11'}V_{00'}\right)\gamma_{0}\theta^{01'}\wedge\theta^{11'}\\
=&\left(V_{10'}V_{01'}-V_{00'}V_{11'}\right)\gamma_{0}\theta^{01'}\wedge\theta^{11'}\\
=&-\Delta_{b}\gamma_{0}\theta^{01'}\wedge\theta^{11'}=0.
\end{split}\end{equation*}
Hence, (\ref{eq:induct}) is solvable for $i=\cdots,-1,0$.

On the other hand, for $i=0,1,\cdots$,
if there exists $\gamma_{i},\gamma_{i+1}$ such that
\begin{equation}\label{eq:ori}
     \left\{\begin{array}{l}
         V_{01'}\gamma_{i}=V_{00'}\gamma_{i+1},\\
         V_{11'}\gamma_{i}=V_{10'}\gamma_{i+1},
         \end{array}  \right.
  \end{equation}
we need to show
\begin{equation}\label{eq:Ward-A1}
     \left\{\begin{array}{l}
         V_{00'}\gamma_{i+2}=V_{01'}\gamma_{i+1},\\
         V_{10'}\gamma_{i+2}=V_{11'}\gamma_{i+1}
         \end{array}  \right.
  \end{equation}
has a solution $\gamma_{i+2}$. Denote
\begin{equation*}
  \widetilde{\Lambda}_{i+1}:=V_{01'}\gamma_{i+1}\theta^{00'}+V_{11'}\gamma_{i+1}\theta^{10'}.
  \end{equation*}
As above $\Lambda_{i+1}$ is $d_{0}$-closed, since
\begin{equation*}
\begin{split}
d_{0}\widetilde{\Lambda}_{i+1}&
=(V_{00'}V_{11'}-V_{10'}V_{01'})\gamma_{i+1}\theta^{00'}\wedge\theta^{10'}=(V_{11'}V_{00'}-V_{01'}V_{10'})\gamma_{i+1}\theta^{00'}\wedge\theta^{10'}\\&
=(V_{11'}V_{01'}-V_{01'}V_{11'})\gamma_{i}\theta^{00'}\wedge\theta^{10'}
=0,
\end{split}
 \end{equation*}
by using $[V_{00'},V_{11'} ]=[V_{10'},V_{01'} ]=T$ in the second identity, (\ref{eq:ori}) and $[V_{11'},V_{01'} ]=0$
in the third identity.
By Lemma \ref{Poin} again, there exists $\gamma_{i+2}$ such that $d_{0}\gamma_{i+2}=\widetilde{\Lambda}_{i+1}$,
i.e. the equation (\ref{eq:Ward-A1}) is solvable. The equation (\ref{eq:induct}), for $i=0,1,\cdots$,
is also reduced to the equation
$$0={\rm d}_{0}\widetilde{\Lambda}_{0}=(V_{00'}V_{11'}-V_{10'}V_{01'})\gamma_{0}\theta^{00'}\wedge\theta^{10'}=\Delta_{b}\gamma_{0}\theta^{00'}\wedge\theta^{10'}.$$
It holds since $\Delta_{b}\gamma_{0}=0$. The lemma is proved.
\end{proof}

{\it Proof of Theorem \ref{ASDA}. (1)}
As in the classical case (cf. e.g. \cite[Example 10.1.2]{Mason1}), we take the Birkhoff decomposition
\begin{equation}\label{cur111000}
  G(\mathbf{y},t,\zeta)=\left(\begin{array} {cc} \zeta&\gamma(\mathbf{y},t,\zeta)\\0&\zeta^{-1} \end{array}\right)=\widetilde{ f}^{-1}f,
 \end{equation}
with \begin{equation}\label{cure112233}\begin{split}
  f&=\frac{1}{\sqrt{\varphi}}\left(\begin{array} {cc} \zeta&\varphi+\gamma_{+}\\-1&-\zeta^{-1}\gamma_{+} \end{array}\right)
  \in\mathcal{O}(U\times \mathbb{C}),\\
\widetilde{ f}&=\frac{1}{\sqrt{\varphi}}\left(\begin{array} {cc} 1&-{\zeta}\gamma_{-}\\-{\zeta}^{-1}&\varphi+\gamma_{-} \end{array}\right)
   \in\mathcal{O}(U\times \mathbb{C}),
\end{split} \end{equation}
where $U\times\mathbb{C}$ and
$U\times\mathbb{C}$ are local coordinate charts of $U\times\mathbb{C}P^{1}$.
Their inverses are
\begin{equation}\label{cure1122330}\begin{split}
  f^{-1}&=\sqrt{\varphi}\left(\begin{array} {cc} -\zeta^{-1}\varphi^{-1}\gamma_{+}&-1-\gamma_{+}\varphi^{-1}\\ \varphi^{-1}&\zeta\varphi^{-1} \end{array}\right),\\
\widetilde{ f}^{-1}&=\sqrt{\varphi}\left(\begin{array} {cc} 1+\varphi^{-1}\gamma_{-}&\zeta\gamma_{-}\varphi^{-1}\\ \varphi^{-1}\zeta^{-1}&\varphi^{-1} \end{array}\right),
 \end{split}\end{equation}
respectively. It is direct to check (\ref{cur111000}) by $\widetilde{ f}^{-1}$ and $f$ in (\ref{cure112233})-(\ref{cure1122330}), if $\gamma_{\pm}$ satisfy (\ref{cure112233000}) and $\gamma_{0}=\varphi$.
Then
\begin{equation}
    h: =f|_{\zeta=0}=\frac{1}{\sqrt{\varphi}}\left(\begin{array} {cc} 0&\varphi\\-1&-\gamma_{-1} \end{array}\right),\qquad
    \widetilde{h}: =\widetilde{ f}|_{\widetilde{\zeta}=0}=\frac{1}{\sqrt{\varphi}}\left(\begin{array} {cc} 1&-\gamma_{1}\\0&\varphi \end{array}\right),
 \end{equation}
  By (\ref{d-1-2}) and (\ref{eq:induct}), we have
${\rm d}_{1}\gamma_{-1}=V_{00'}\varphi\cdot\theta^{01'}+V_{10'}\varphi\cdot\theta^{11'}$ and ${\rm d}_{0}\gamma_{1}=V_{01'}\varphi\cdot\theta^{00'}+V_{11'}\varphi\cdot\theta^{10'}$, and
\begin{equation}
    h^{-1}=\sqrt{\varphi}\left(\begin{array} {cc} -\varphi^{-1}\gamma_{-1}&-1\\ \varphi^{-1}&0 \end{array}\right),\qquad
    \widetilde{h}^{-1}=\sqrt{\varphi}\left(\begin{array} {cc} 1&\gamma_{1}\varphi^{-1}\\0&\varphi^{-1} \end{array}\right).
 \end{equation}
It is direct to check that
 \begin{equation}\label{ASDA1}    \begin{split}
     \Phi &
     =-{\rm d}_{1}{h} \cdot h^{-1}-{\rm d}_{0}{\widetilde{h}} \cdot \widetilde{h}^{-1}+\Phi_{T}\theta\\&
     =\left[\begin{array} {cc}
     \frac{1}{2}(-{\rm d}_{1}\ln{\varphi}+{\rm d}_{0}\ln{\varphi})
     & V_{01'}(\ln\varphi)\theta^{00'}+V_{11'}(\ln\varphi)\theta^{10'}\\
     V_{00'}(\ln\varphi)\theta^{01'}+V_{10'}(\ln\varphi)\theta^{11'}
     &\  \frac{1}{2}({\rm d}_{1}\ln{\varphi}-{\rm d}_{0} \ln{\varphi})\end{array}\right]+\Phi_{T}\theta,
\end{split}     \end{equation}
which is ASD by (\ref{eq:asd-eq4}) and  Theorem \ref{pwc}.

(2)\quad
Noting that $V_{AA'}(y_{BB'})=\delta_{AB}\delta_{A'B'}$ and $\left\|\mathbf{ y}\right\|^{2}=y_{00'}y_{11'}-y_{10'}y_{01'}$, we have
 $$V_{00'}\left\|\mathbf{ y}\right\|^{2}=y_{11'},\quad V_{01'}\left\|\mathbf{ y}\right\|^{2}=-y_{10'},\quad
 V_{10'}\left\|\mathbf{ y}\right\|^{2}=-y_{01'},\quad V_{11'}\left\|\mathbf{ y}\right\|^{2}=y_{00'}.$$
 It's direct to check that
\begin{equation*}\begin{split}
&V_{00'}(\left\|\mathbf{ y}\right\|^{4}-t^{2})=2y_{11'}(\left\|\mathbf{ y}\right\|^{2}+t),\\
&V_{01'}(\left\|\mathbf{ y}\right\|^{4}-t^{2})=-2y_{10'}(\left\|\mathbf{ y}\right\|^{2}+t),\\
&V_{10'}(\left\|\mathbf{ y}\right\|^{4}-t^{2})=2y_{01'}(-\left\|\mathbf{ y}\right\|^{2}+t),\\
&V_{11'}(\left\|\mathbf{ y}\right\|^{4}-t^{2})=2y_{00'}(\left\|\mathbf{ y}\right\|^{2}-t),\\
\end{split}\end{equation*}
and
\begin{equation*}\begin{split}
&V_{10'}V_{01'}(\left\|\mathbf{ y}\right\|^{4}-t^{2})=-2\left\|\mathbf{ y}\right\|^{2}-2t+4y_{10'}y_{01'},\\
&V_{00'}V_{11'}(\left\|\mathbf{ y}\right\|^{4}-t^{2})=2\left\|\mathbf{ y}\right\|^{2}-2t+4y_{00'}y_{11'}.\\
\end{split}\end{equation*}
Consequently, for $\varphi=\frac{1}{\left\|\mathbf{ y}\right\|^{4}-t^{2}}$, we have
\begin{equation*}\begin{split}
V_{10'}V_{01'}(\varphi)&=-\frac{V_{10'}V_{01'}(\left\|\mathbf{ y}\right\|^{4}-t^{2})}{(\left\|\mathbf{ y}\right\|^{4}-t^{2})^{2}}
+2\frac{V_{01'}(\left\|\mathbf{ y}\right\|^{4}-t^{2})V_{10'}(\left\|\mathbf{ y}\right\|^{4}-t^{2})}{(\left\|\mathbf{ y}\right\|^{4}-t^{2})^{3}}\\
&=\frac{2y_{00'}y_{11'}+2y_{10'}y_{01'}+2t}{(\left\|\mathbf{ y}\right\|^{4}-t^{2})^{2}},\\
V_{00'}V_{11'}(\varphi)&=-\frac{V_{00'}V_{11'}(\left\|\mathbf{ y}\right\|^{4}-t^{2})}{(\left\|\mathbf{ y}\right\|^{4}-t^{2})^{2}}
+2\frac{V_{11'}(\left\|\mathbf{ y}\right\|^{4}-t^{2})V_{00'}(\left\|\mathbf{ y}\right\|^{4}-t^{2})}{(\left\|\mathbf{ y}\right\|^{4}-t^{2})^{3}}\\
&=\frac{2y_{00'}y_{11'}+2y_{10'}y_{01'}+2t}{(\left\|\mathbf{ y}\right\|^{4}-t^{2})^{2}}.\\
\end{split}\end{equation*}
Therefore $V_{10'}V_{01'}(\varphi)=V_{00'}V_{11'}(\varphi),$ i.e. $\varphi$ is the solution to  (\ref{eq:compatible-}).\qed

\section{The Real Case }
\subsection{SD and ASD horizontal form on real Heisenberg group $\mathscr H^{\mathbb{R}}$ }
The real Heisenberg group is a flat model of contact manifolds. Let us compare the ASD connection restricted to the real $5$D Heisenberg group with contact instantons on $5$D contact manifolds introduced by physicists
(cf.\cite{MR3042942} \cite{MR2967134} \cite{Wolf1}).
Take a metric on $\mathscr H^{\mathbb{R}}$ as
$g=\rm{d}y_{1}^{2}+\rm{d}y_{2}^{2}+\rm{d}y_{3}^{2}+\rm{d}y_{4}^{2}+
\rm{d}t^{2}.$
The relevant volume element on $\mathscr H^{\mathbb{R}}$ is
$\rm{d}V:=\rm{d}y_{1}\wedge \rm{d}y_{2}\wedge \rm{d}y_{3}\wedge \rm{d}y_{4}\wedge \rm{d}t.$
For briefness, we denote $\rm{d}y_{5}:=\rm{d}t$. The {\it Hodge star} $*$ with respect to the metric $g$ on $\mathscr H^{\mathbb{R}}$ is given by $*:\Omega^{k}(\mathscr H^{\mathbb{R}})\longrightarrow \Omega^{5-k}(\mathscr H^{\mathbb{R}})$, $ w\longmapsto *w$, such that
\begin{equation}\label{Hodg1}
w\wedge*w=\rm{d}V.
\end{equation}
Namely, we have
$*(\rm{d}y_{i_{1}}\wedge\cdots\wedge\rm{d}y_{i_{k}})=\varepsilon_{i_{1}\cdots i_{k}j_{1}\cdots j_{5-k}}\rm{d}y_{j_{1}}\wedge\cdots\rm{d}y_{j_{5-k}},$
where $\varepsilon_{i_{1}\cdots i_{k}j_{1}\cdots j_{5-k}}$ is the permutation index from $\{i_{1},\cdots, i_{k},j_{1},\cdots, j_{5-k}\}$ to $\{1,\cdots,5\}$.
 Denote by $\Omega^{2}_{H}(\mathscr H^{\mathbb{R}})$ the space of horizontal two forms on $\mathscr H^{\mathbb{R}}$, i.e. $\iota_{T}\omega=0$ for $\omega\in\Omega^{2}_{H}(\mathscr H^{\mathbb{R}})$, and $\Omega^{2}_{V}(\mathscr H^{\mathbb{R}})$ the space of vertical
two forms on $\mathscr H^{\mathbb{R}}$. We have the decomposition
\begin{equation}
\Omega^{2}(\mathscr H^{\mathbb{R}})=\Omega^{2}_{H}(\mathscr H^{\mathbb{R}})\oplus\Omega^{2}_{V}(\mathscr H^{\mathbb{R}}),\qquad
\omega=\omega_{H}+\omega_{V},
\end{equation}
where $\omega_{H}=\iota_{T}({\rm dt}\wedge\omega)$ and $\omega_{V}={\rm dt}\wedge(\iota_{T}\omega)$.
Denote by $\Omega^{2+}_{H}(\mathscr H^{\mathbb{R}})$ $\left(\Omega^{2-}_{H}(\mathscr H^{\mathbb{R}})\right)$ the space of all horizontal (anti-)self-dual spaces of two forms on $\mathscr H$, of which elements satisfy
\begin{equation}\label{hsdu}\iota_{T}\ast\omega=\omega\quad (\iota_{T}\ast\omega=-\omega).\end{equation}
Then we have the decomposition
\begin{equation}
\Omega^{2}_{H}(\mathscr H^{\mathbb{R}})=\Omega^{2+}_{H}(\mathscr H^{\mathbb{R}})\oplus\Omega^{2-}_{H}(\mathscr H^{\mathbb{R}})\qquad
\omega_{H}=\omega^{+}_{H}+\omega^{-}_{H},
\end{equation}
where $\omega^{+}_{H}=\frac{1}{2}(1+\iota_{T}\ast)\omega_{H}\in\Omega^{2+}_{H}(\mathscr H)$ and $\omega^{-}_{H}=\frac{1}{2}(1-\iota_{T}\ast)\omega_{H}\in\Omega^{2-}_{H}(\mathscr H^{\mathbb{R}})$. In fact, noting that $\iota_{T}\ast\iota_{T}\ast=id,$ we have $\iota_{T}\ast\omega^{+}_{H}=\frac{1}{2}\iota_{T}\ast(1+\iota_{T}\ast)\omega_{H}=\frac{1}{2}(\iota_{T}\ast+1)\omega_{H}=\omega^{+}_{H},$ i.e. $\omega^{+}_{H}\in\Omega^{2+}_{H}(\mathscr H)$. And
$\iota_{T}\ast\omega^{-}_{H}=\frac{1}{2}\iota_{T}\ast(1-\iota_{T}\ast)\omega_{H}=\frac{1}{2}(\iota_{T}\ast-1)\omega_{H}=-\omega^{-}_{H},$ i.e. $\omega^{-}_{H}\in\Omega^{2-}_{H}(\mathscr H)$.

\subsection{Contact instantons on $\mathscr H^{\mathbb{R}}$}

The double fibration over real Heisenberg group $\mathscr H^{\mathbb{R}}$ becomes
\begin{equation}  \label{rfibration}      \xy 0;/r.22pc/:
 (0,20)*+{\mathcal{F}^{\mathbb{R}} }="1";
(-20,0)*+{  \mathcal{P}^{\mathbb{R}} }="2";
(20,0)*+{ \mathscr H^{\mathbb{R}}}="3";
{\ar@{->}_{\eta}"1"; "2"};
{\ar@{->}^{\tau} "1"; "3"};
\endxy
 \end{equation}
 where $\mathcal{F^{\mathbb{R}}}=\mathscr H^{\mathbb{R}}\times \mathbb{C}P^1$ and $\mathcal{P^{\mathbb{R}}}=\eta\left(\mathscr H^{\mathbb{R}}\times\mathbb{C}P^{1}\right)$.
\begin{cor}
\label{prop:YX1}   The fiber of the mapping $\eta:\mathbb{R}^{5}\times \mathbb{C}\rightarrow \mathcal{P}^{\mathbb{R}}$ given  by
\begin{equation}\label{eq:psi21}\begin{split}
\omega_{0}:=\eta_{0}(\mathbf{y}_{0'}, \mathbf{ y}_{1'},\mathbf{t},\zeta)&=y_{1}+\mathbf{i}y_{2}+\zeta (-y_{3}+\mathbf{i}y_{4}),\\
\omega_{1}:=\eta_{1}(\mathbf{y}_{0'}, \mathbf{ y}_{1'},\mathbf{t},\zeta)&=y_{3}+\mathbf{i}y_{4}+\zeta (y_{1}-\mathbf{i}y_{2}),\\
\omega_{2}:=\eta_{2}(\mathbf{y}_{0'}, \mathbf{ y}_{1'},\mathbf{t},\zeta)&
=-\mathbf{i}s-2\zeta(-y_{3}+\mathbf{i}y_{4})(y_{1}-\mathbf{i}y_{2})- y_{1}^2 - y_{2}^2
 +  y_{3}^2+ y_{4}^2,\\
\omega_{3}:=\eta_{3}(\mathbf{y}_{0'}, \mathbf{ y}_{1'},\mathbf{t},\zeta)&= \zeta,\\
\end{split}\end{equation}
 is an abelian subgroup of dimension $4$, whose tangential space is spanned by $\{V_{0},V_{1}\}$. We denote $\hat{y}=(\omega_{0},\omega_{1},\omega_{2},\omega_{3})$.
\end{cor}
\begin{prop}
For x and y $\in$ $\mathbb{R}^{5}$, if $\hat{x}\bigcap\hat{y}\neq\emptyset$, then we have $x=y$. So $\mathcal{P}^{\mathbb{R}}$ is the trivial $\mathbb{C}P^{1}$ bundle over $\mathbb{R}^{5}$.
\end{prop}
\begin{proof}
If we write \begin{equation}\label{sinlonly}\begin{split}
&x=(x_{00'},x_{10'},x_{01'},x_{11'},t_{1})=(x_{1}+\textbf{i}x_{2},x_{3}+\textbf{i}x_{4},-x_{3}+\textbf{i}x_{4},x_{1}-\textbf{i}x_{2},-\textbf{i}s_{1}),\\
&y=(y_{00'},y_{10'},y_{01'},y_{11'},t_{2})=(y_{1}+\textbf{i}y_{2},y_{3}+\textbf{i}y_{4},-y_{3}+\textbf{i}y_{4},y_{1}-\textbf{i}y_{2},-\textbf{i}s_{2}),
\end{split}\end{equation}by the embedding (\ref{eq:C-L}) of $\mathscr H^{\mathbb{R}}$ into $\mathscr H$,
then we have
\begin{equation}\label{sinlonly11}\begin{split}
&\hat{x}=\eta\circ\tau^{-1}(x)=\left(x_{00'}+\zeta x_{01'},x_{10'}+\zeta x_{11'},-\textbf{i}s_{1}-\langle \mathbf{x}_{0'}+\zeta\mathbf{x}_{1'},\mathbf{x}_{1'}\rangle,\zeta\right) ,\\
&\hat{y}=\eta\circ\tau^{-1}(y)=\left(y_{00'}+\zeta y_{01'},y_{10'}+\zeta y_{11'},-\textbf{i}s_{2}-\langle \mathbf{y}_{0'}+\zeta\mathbf{y}_{1'},\mathbf{y}_{1'}\rangle,\zeta\right) ,\\
\end{split}\end{equation}
by (\ref{eq:psi}). If $\hat{x}\bigcap\hat{y}\neq\emptyset$, then there exists $\zeta$ such that     equations
\begin{equation}\label{eq:linear eq3}
     \left\{\begin{array}{l}
x_{00'}+\zeta x_{01'}=y_{00'}+\zeta y_{01'},\\
x_{10'}+\zeta x_{11'}=y_{10'}+\zeta y_{11'},\\
-\textbf{i}s_{1}-\langle \mathbf{x}_{0'}+\zeta\mathbf{x}_{1'},\mathbf{x}_{1'}\rangle=-\textbf{i}s_{2}-\langle \mathbf{y}_{0'}+\zeta\mathbf{y}_{1'},\mathbf{y}_{1'}\rangle
         \end{array} \right.
  \end{equation}
must have a solution. So      equations
 \begin{equation*}
\begin{pmatrix}
x_{00'}-y_{00'}&x_{01'}-y_{01'}\\
x_{10'}-y_{10'}&x_{11'}-y_{11'}
\end{pmatrix}
\begin{pmatrix}
a_{1}\\
a_{2}
\end{pmatrix}=0
\end{equation*}
has a nontrivial solution. Consequently, its coefficient matrix is singular, i.e.
\begin{equation*}\begin{split}
0&=\left|\begin{array}{cccc}
x_{00'}-y_{00'}&x_{01'}-y_{01'}\\
x_{10'}-y_{10'}&x_{11'}-y_{11'}
\end{array}\right|
=\left|\begin{array}{cccc}
x_{1}-y_{1}+\textbf{i}x_{2}-\textbf{i}y_{2}&-x_{3}+y_{3}+\textbf{i}x_{4}-\textbf{i}y_{4}\\
x_{3}-y_{3}+\textbf{i}x_{4}-\textbf{i}y_{4}&x_{1}-y_{1}-\textbf{i}x_{2}+\textbf{i}y_{2}
\end{array}\right|\\
&=(x_{1}-y_{1})^{2}+(x_{2}-y_{2})^{2}+(x_{3}-y_{3})^{2}+(x_{4}-y_{4})^{2},
\end{split}
\end{equation*}
i.e., $x_{k}=y_{k}$ for $k=1,\cdots,4$. So $s_{1}=s_{2}$ in (\ref{sinlonly}) by the third equation of (\ref{eq:linear eq3}). This implies $\eta$ in (\ref{rfibration}) is one to one, and so a diffeomorphism. Therefore, $\mathcal{P}^{\mathbb{R}}$ is a topologically  trivial $\mathbb{C}P^{1}$ bundle over $\mathbb{R}^{5}$.
\end{proof}

Recall the real embedding (\ref{eq:C-L}) of $\mathscr H^{\mathbb{R}}$ into $\mathscr H$.
 Then we have the left invariant complex vector field over $\mathscr H^{\mathbb{R}}$
 \begin{equation} \label{eq:realY}\begin{split}
V_{00'} &=\frac{1}{2}(\partial_{y_{1}}-\textbf{i}\partial_{y_{2}})-\textbf{i}{(y_{1}-\textbf{i}y_{2})}\partial_{s},\quad
V_{01'} =\frac{1}{2}(-\partial_{y_{3}}-\textbf{i}\partial_{y_{4}})+\textbf{i}{(y_{3}+\textbf{i}y_{4})}\partial_{s},\\
V_{10'} &=\frac{1}{2}(\partial_{y_{3}}-\textbf{i}\partial_{y_{4}})+\textbf{i}{(y_{3}-\textbf{i}y_{4})}\partial_{s},\quad
V_{11'} =\frac{1}{2}{(\partial_{y_{1}}+\textbf{i}\partial_{y_{2}})}+\textbf{i}{(y_{1}+\textbf{i}y_{2})}\partial_{s},\qquad
T =\textbf{i}\partial{s},
\end{split}\end{equation}
and the relevant dual forms becomes
\begin{equation}\label{rform}
\begin{split}&
\theta^{00'}={\rm{d}}y_{1}+\textbf{i}{\rm{d}}y_{2},\quad\theta^{01'}=-{\rm{d}}y_{3}+\textbf{i}{\rm{d}}y_{4},\quad
\theta^{10'}={\rm{d}}y_{3}+\textbf{i}{\rm{d}}y_{4},\quad
\theta^{11'}={\rm{d}}y_{1}-\textbf{i}{\rm{d}}y_{2},\\&
\theta=-\textbf{i}{\rm{d}}s+2\textbf{i}(y_{1}{\rm{d}}y_{2}-y_{2}{\rm{d}}y_{1}+y_{4}{\rm{d}}y_{3}-y_{3}{\rm{d}}y_{4}).
\end{split}
\end{equation}

\begin{prop}
The space of horizontal self-dual $2$-forms on $\mathscr H^{\mathbb{R}}$ is spanned by $\{S^{0'0'},S^{0'1'}$, $S^{1'1'}\}$ with
\begin{equation}
S^{0'0'}=\theta^{00'}\wedge\theta^{10'},\quad
S^{0'1'}=\theta^{00'}\wedge\theta^{11'}-\theta^{10'}\wedge\theta^{01'},\quad
S^{1'1'}=\theta^{01'}\wedge\theta^{11'}.
\end{equation}
The space of horizontal anti-self-dual 2-forms is spanned by $\{S^{00},S^{01},S^{11}\}$ with
\begin{equation}
S^{00}=\theta^{00'}\wedge\theta^{01'},\quad
S^{01}=\theta^{00'}\wedge\theta^{11'}+\theta^{10'}\wedge\theta^{01'},\quad
S^{11}=\theta^{10'}\wedge\theta^{11'}.
\end{equation}
\end{prop}

 The  curvature of the connection form $\Phi=\Phi_{AB'}\theta^{AB'}+\Phi_{T}\theta$ is
\begin{equation}\label{eq:lienar Rcurvature}
\begin{split}
F=&d\Phi+\Phi\wedge\Phi\\
=&d(\Phi_{AB'})\wedge\theta^{AB'}+d(\Phi_{T})\wedge\theta+\Phi_{T}d\theta+\Phi_{AB'}\Phi_{CD'}\theta^{AB'}\wedge\theta^{CD'}+[\Phi_{AB'},\Phi_{T}]\theta^{AB'}\wedge\theta\\
=&V_{CD'}(\Phi_{AB'})\theta^{CD'}\wedge\theta^{AB'}+T(\Phi_{AB'})\theta\wedge\theta^{AB'}+V_{AB'}(\Phi_{T})\theta^{AB'}\wedge\theta+\Phi_{T}d\theta\\
&+\Phi_{AB'}\Phi_{CD'}\theta^{AB'}\wedge\theta^{CD'}+[\Phi_{AB'},\Phi_{T}]\theta^{AB'}\wedge\theta\\
=&[V_{AB'}(\Phi_{CD'})+\Phi_{AB'}\Phi_{CD'}]\theta^{AB'}\wedge\theta^{CD'}+\Phi_{T}d\theta\\
&+(V_{AB'}(\Phi_{T})-T(\Phi_{AB'})+[\Phi_{AB'},\Phi_{T}])\theta^{AB'}\wedge\theta,
         \end{split}
  \end{equation}
by relabeling indices. Here we use the Einstein convention of summation over repeated indices.
Then we have
\begin{equation}\label{eq:lienar RcurvatureV}
F_{V}=\theta\wedge\iota_{T}F=(V_{AB'}(\Phi_{T})-T(\Phi_{AB'})+[\Phi_{AB'},\Phi_{T}])\theta^{AB'}\wedge\theta,
  \end{equation}
  and
\begin{equation}\label{eq:lienar RcurvatureH}
\begin{split}
F_{H}=&\iota_{T}(\theta\wedge F)=[V_{AB'}(\Phi_{CD'})+\Phi_{AB'}\Phi_{CD'}]\theta^{AB'}\wedge\theta^{CD'}+\Phi_{T}d\theta\\
=&(V_{00'}\Phi_{01'}-V_{01'}\Phi_{00'}+[\Phi_{00'},\Phi_{01'}])\theta^{00'}\wedge\theta^{01'}\\
&+(V_{00'}\Phi_{10'}-V_{10'}\Phi_{00'}+[\Phi_{00'},\Phi_{10'}])\theta^{00'}\wedge\theta^{10'}\\
&+(V_{00'}\Phi_{11'}-V_{11'}\Phi_{00'}+[\Phi_{00'},\Phi_{11'}]-2\Phi_{T})\theta^{00'}\wedge\theta^{11'}\\
&+(V_{01'}\Phi_{10'}-V_{10'}\Phi_{01'}+[\Phi_{01'},\Phi_{10'}]+2\Phi_{T})\theta^{01'}\wedge\theta^{10'}\\
&+(V_{01'}\Phi_{11'}-V_{11'}\Phi_{01'}+[\Phi_{01'},\Phi_{11'}])\theta^{01'}\wedge\theta^{11'}\\
&+(V_{10'}\Phi_{11'}-V_{11'}\Phi_{10'}+[\Phi_{10'},\Phi_{11'}])\theta^{10'}\wedge\theta^{11'}.
         \end{split}
  \end{equation}
So its horizontal self-dual parts is
  \begin{equation}\label{eq:lienar RcurvatureH^{+}}
\begin{split}
F_{H}^{+}
=&(V_{00'}\Phi_{10'}-V_{10'}\Phi_{00'}+[\Phi_{00'},\Phi_{10'}])S^{0'0'}
+(V_{01'}\Phi_{11'}-V_{11'}\Phi_{01'}+[\Phi_{01'},\Phi_{11'}])S^{1'1'}\\
&+\frac{1}{2}\left(V_{00'}\Phi_{11'}-V_{11'}\Phi_{00'}+[\Phi_{00'},\Phi_{11'}]
+V_{01'}\Phi_{10'}-V_{10'}\Phi_{01'}+[\Phi_{01'},\Phi_{10'}]\right)S^{0'1'}
.
         \end{split}
  \end{equation}
As a result, we get that
\begin{prop}
(1)$F_{H}^{+}=0$ is equivalent to $F(V_{0},V_{1})=0$,i.e.,
\begin{equation}\label{eq:RASD}
     \left\{\begin{array}{l}
V_{00'}(\Phi_{10'})-V_{10'}(\Phi_{00'})+[\Phi_{00'},\Phi_{10'}]=0,\\
V_{01'}(\Phi_{10'})+V_{00'}(\Phi_{11'})-V_{10'}(\Phi_{01'})-V_{11'}(\Phi_{00'})+[\Phi_{00'},\Phi_{11'}]+[\Phi_{01'},\Phi_{10'}]=0,\\
V_{01'}(\Phi_{11'})-V_{11'}(\Phi_{01'})+[\Phi_{01'},\Phi_{11'}]=0.
         \end{array}   \right.
 \end{equation}
 (2)  $F_{V}=0$ is equivalent to
$V_{AB'}(\Phi_{T})-T(\Phi_{AB'})+[\Phi_{AB'},\Phi_{T}]=0$.
\end{prop}

When restricted to the real Heisenberg group $\mathscr H^{\mathbb{R}}$, the {\it Sub-Laplacian} becomes
$$\Delta_{b}:=(X_{1}^{2}+X_{2}^{2}+X_{3}^{2}+X_{4}^{2}),$$
where
 \begin{equation}\begin{split}
X_{1} &=\frac{1}{2}\partial_{x_{1}}+ x_{2}\partial_{s},\quad
X_{2} =\frac{1}{2}\partial_{x_{2}}- x_{1}\partial_{s},\quad
X_{3} =\frac{1}{2}\partial_{x_{3}}+x_{4}\partial_{s},\quad
X_{4} =\frac{1}{2}\partial_{x_{4}}- x_{3}\partial_{s}.
\end{split}\end{equation}
Note that $[X_{1},X_{2}]=-\partial_{s}$, $[X_{3},X_{4}]=-\partial_{s}$.
Then
\begin{equation*}
\varphi=\frac{1}{|x|^{4}+s^{2}},\qquad {\rm  where} \quad |x|=(x_{1}^{2}+x_{2}^{2}+x_{3}^{2}+x_{4}^{2})^{\frac{1}{2}},
\end{equation*}
is the solution to the Sub-Laplacian equation
\begin{equation}\label{eq:lapr}
\Delta_{b}\varphi=0.
\end{equation}

\begin{cor}
If $\varphi$ is the solution to the equation (\ref{eq:lapr}), then the connection form $\Phi$ in (\ref{AASSDD11}) restricted to the real Heisenberg group $\mathscr H^{\mathbb{R}}$ satisfies the horizontal part of the ASD contact instanton equation $F_{H}^{+}=0$.
\end{cor}

\begin{rem}\label{5kc} By Penrose-Ward correspondence, Wolf \cite{Wolf1} has already characterized
the solution to the horizontal part of the    ``self-dual" contact instanton equation on a contact manifold as follows.
Let $M$ be a $5$D $K$-contact manifold with Cauchy-Riemann twistor space $\pi:Z\rightarrow M$   and integrable {\it Cauchy-Riemann} structure. There is a one-to-one correspondence between \\
(i) rank-r Cauchy-Riemann vector bundles $E_{Z}\rightarrow Z$ such that the restriction $E_{Z}|_{\pi^{-1}(p)}$ is holomorphically trivial for all $p\in M$ and\\
(ii) rank-r complex vector bundles $E_{M}\rightarrow M$ equipped with a connection $\nabla$ and curvature $F=\nabla^{2}$ such that the projection on the contact distribution is $F_{H}\in \Omega^{2}_{+}(M, End E_{M})$, that is, $F_{-}=0$.
\end{rem}

\appendix

\section{Double Fibration and Twistor Transform }
Let us use the method in {\cite{Wa13}} to write down local coordinates of the double fibration
 \begin{equation}\label{Ptwistor1}        \xy 0;/r.22pc/:
 (0,20)*+{{\rm SO}(6,\mathbb{C})/R }="1";
(-20,0)*+{{\rm SO}(6,\mathbb{C})/Q}="2";
(20,0)*+{{\rm SO}(6,\mathbb{C})/P}="3";
{\ar@{->}_{\eta}"1"; "2"};
{\ar@{->}^{\tau} "1"; "3"};
\endxy
 \end{equation}
 and the mappings $\eta$ and $\tau$ in term of their local coordinates, where $P$, $Q$ and $R$ are subgroups of ${\rm SO}(6,\mathbb{C})$ in the following. Here the mapping $\eta$ and $\tau$ are given by
 $$\eta(gR)=gQ,\qquad \tau(gR)=gP.$$
${\rm SO}(6,\mathbb{C})$ is the group of all matrices preserving the matrices ${\rm I}=\left(\begin{array} {cc} 0&E_{3}\\E_{3}&0 \end{array}\right)$, where $E_{3}$ is $3\times3$ identity matrix, i.e. all $X\in {\rm GL}(6,\mathbb{C})$ such that $X^{t}IX=I$.
The Lie algebra ${\rm so}(6,\mathbb{C})$ of ${\rm SO}(6,\mathbb{C})$ consists of all $X\in {\rm gl}(6,\mathbb{C})$ such that $X^{t}I+IX=0$ (cf. \cite{FH}     \cite{H}), i.e.
 $${\rm so}(6,\mathbb{C})=\left\{\left[\begin{array} {cc} A&B\\C&-A^t \end{array}\right];B^t=-B, C^t=-C,
 A,B,C\in {\rm gl}(3,\mathbb{C})\right\}.$$

Consider the subalgebra of ${\rm so}(6, \mathbb{C})$,
\begin{equation*}
 \begin{split}
\mathfrak{p}&= \left\{\begin{bmatrix}
     *&*&*&0&*&*\\
     *&*&*&*&0&*\\
     0&0&*&*&*&0\\
     0&0&0&*&*&0\\
     0&0&0&*&*&0\\
     0&0&0&*&*&*\end{bmatrix}\in {\rm so}(6, \mathbb{C})\right\},
\mathfrak{h}= \left\{\begin{bmatrix}
     0&0&0&0&0&0\\
     0&0&0&0&0&0\\
     y_{00'}&y_{01'}&0&0&0&0\\
     0&-\frac{1}{2}t&-y_{10'}&0&0& -y_{00'}\\
     \frac{1}{2}t&0&-y_{11'}&0&0& -y_{01'} \\
     y_{10'}&y_{11'}&0&0&0&0\end{bmatrix}\in {\rm so}(6, \mathbb{C})\right\}.
       \end{split}
 \end{equation*}
 We have the decomposition  ${\rm so}(6, \mathbb{C})=\mathfrak{h}\oplus \mathfrak{p}.$
 Similarly, we have the decomposition  ${\rm so}(6, \mathbb{C})=\mathfrak{n}\oplus \mathfrak{q}$  with
 \begin{equation*}
 \begin{split}
 \mathfrak{q}&= \left\{\begin{bmatrix}
     *&*&*&0&*&*\\
     0&*&*&*&0&*\\
     0&*&*&*&*&0\\
     0&0&0&*&0&0\\
     0&0&*&*&*&*\\
     0&*&0&*&*&*\end{bmatrix}\in {\rm so}(6, \mathbb{C})\right\},\quad
     \mathfrak{n}= \left\{\begin{bmatrix}
     0&0&0&0&0&0\\
     y_{1}&0&0&0&0&0\\
     y_{2}&0&0&0&0&0\\
     0&-y_{3}&-y_{4}&0&-y_{1}& -y_{2}\\
     y_{3}&0&0&0&0&0\\
     y_{4}&0&0&0&0&0\end{bmatrix}\in {\rm so}(6, \mathbb{C})\right\},
      \end{split}
 \end{equation*}
and $\rm{so}(6, \mathbb{C})=\mathfrak{m}\oplus \mathfrak{r}, $ with
 \begin{equation*}
 \begin{split}
 \mathfrak{r}&=\mathfrak{q}\cap\mathfrak{p}=  \left\{\begin{bmatrix}
     *&*&*&0&*&*\\
     0&*&*&*&0&*\\
     0&0&*&*&*&0\\
     0&0&0&*&0&0\\
     0&0&0&*&*&0\\
     0&0&0&*&*&*\end{bmatrix}\in {\rm so}(6, \mathbb{C})\right\},\quad
     \mathfrak{m}= \left\{\begin{bmatrix}
     0&0&0&0&0&0 \\
     *&0&0&0&0&0 \\
     *&*&0&0&0&0 \\
     0&*&*&0&*&*\\
     *&0&*&0&0&*\\
     *&*&0&0&0&0\end{bmatrix}\in {\rm so}(6, \mathbb{C})\right\}.
      \end{split}
 \end{equation*}
Take $P$, $Q$ and $R$ to be Lie groups with Lie algebra to be $\mathfrak{p}$, $\mathfrak{q}$ and $\mathfrak{r}$, respectively.

To write down the double fibration concretely, let us recall the model case ${\rm SL}(2,\mathbb{C})/B_{0}$ as in \cite{Wa13} with the parabolic subgroup
   \begin{equation*}
 B_{0}=\left\{\left[\begin{array} {cc} a&b\\0&c\end{array}\right];a,b,c\in\mathbb{C}\  { \rm and }\ ac\neq0\right\}.
  \end{equation*}
The homogeneous space  $SL(2,\mathbb{C})/B_{0}$ is covered by two coordinate charts
$\iota_{1}:\mathbb{C}\longrightarrow SL(2,\mathbb{C})/B_{0}$ given by
$z\longmapsto M_{z}B_{0}$ and
$\iota_{2}:\mathbb{C}\longrightarrow SL(2,\mathbb{C})/B_{0}$  given by
$z\longmapsto W_{0}M_{\widetilde{z}}B_{0}$,
where
$
M_{z}=\left[\begin{array} {cc} 1&0\\z&1\end{array}\right],$ $
W_{0}=\left[\begin{array} {cc} 0&1\\-1&0\end{array}\right]$ is the Weyl element.
Since $W_{0}M_{\widetilde{z}}\in\left[\begin{array} {cc} 1&0\\-\frac{1}{\widetilde{z}}&1 \end{array}\right]B_{0}$,
 the transition function of these two coordinate charts is
\begin{equation*}
\begin{split}
\iota_{1}^{-1}\circ\iota_{2}:\mathbb{C}\setminus\{0\}&\longrightarrow\mathbb{C}\setminus\{0\},\quad
\widetilde{z}\longmapsto-\frac{1}{\widetilde{z}}.
\end{split}
  \end{equation*}
So we get ${\rm SL}(2,\mathbb{C})/B_{0}\cong\mathbb{C}P^{1}$.

 For an element $Y$ of $\mathfrak{h}$ given by
  $$Y=\begin{bmatrix}
     0&0&0&0&0&0\\
     0&0&0&0&0&0\\
     y_{00'}&y_{01'}&0&0&0&0\\
     0&-\frac{1}{2}t&-y_{10'}& 0&0& -y_{00'}\\
     \frac{1}{2}t&0&-y_{11'}&0&0& -y_{01'}\\
     y_{10'}&y_{11'}&0&0&0&0\end{bmatrix},$$
it's direct to see that
$$Y^{2}=\begin{bmatrix}
     0&0&0&0&0&0\\
     0&0&0&0&0&0\\
     0&0&0&0&0&0\\
     -\langle\mathbf{y}_{0'},\mathbf{y}_{0'}\rangle&-\langle\mathbf{y}_{0'},\mathbf{y}_{1'}\rangle&0&0&0&0\\
     -\langle\mathbf{y}_{0'},\mathbf{y}_{1'}\rangle&-\langle\mathbf{y}_{1'},\mathbf{y}_{1'}\rangle&0&0&0&0\\
     0&0&0&0&0&0\end{bmatrix},\qquad Y^{3}=0.$$
So we have
 \begin{equation}\label{eq:Hisen}
e^{Y}=\begin{bmatrix}
     1&0&0&0&0&0\\
     0&1&0&0&0&0\\
     y_{00'}&y_{01'}&1&0&0&0\\
     -\frac{1}{2}\langle\mathbf{y}_{0'},\mathbf{y}_{0'}\rangle&-\frac{1}{2}t-\frac{1}{2}\langle\mathbf{y}_{0'},\mathbf{y}_{1'}\rangle&-y_{10'}& 1&0& -y_{00'}\\
     \frac{1}{2}t-\frac{1}{2}\langle\mathbf{y}_{0'},\mathbf{y}_{1'}\rangle&-\frac{1}{2}\langle\mathbf{y}_{1'},\mathbf{y}_{1'}\rangle&-y_{11'}&0&1& -y_{01'}\\
     y_{10'}&y_{11'}&0&0&0&1\end{bmatrix},
\end{equation}
which we denote by $H_{(\mathbf{y},t)}.$
Then we see that $
     \mathcal H=  \{H_{(\mathbf{y},t)}|\mathbf{y}\in\mathbb{C}^{4},$ $t\in\mathbb{C} \}
 $
is the Lie group   with Lie algebra $\mathfrak{h}$.

 \begin{lem}
 $\mathcal H$ is a group, and $H:\mathscr{H}\rightarrow \mathcal H,$ given by
$
 (\mathbf{y},\mathbf{t})\mapsto H_{(\mathbf{y},\mathbf{t})}
$
 is an isomorphism.
 \end{lem}
 \begin{proof}
 For $(\mathbf{y},t)$, $(\mathbf{\hat{y}},\hat{t})$ $\in\mathscr{H}$, it is direct to check that
  \begin{equation}
 \begin{split}&
 H_{(\mathbf{y},\mathbf{t})}\cdot H_{(\mathbf{\hat{y}},\mathbf{\hat{t}} )}=
 \begin{bmatrix}
     1&0&0&0&0&0\\
     0&1&0&0&0&0\\
     y_{00'}+\hat{y}_{00'}&y_{01'}+\hat{y}_{01'}&1&0&0&0\\
     \widetilde{t}_{00'}& \widetilde{t}_{01'}&y_{10'}+\hat{y}_{10'}& 1&0&-(y_{00'}+\hat{y}_{00'})\\
     \widetilde{t}_{10'}&\widetilde{t}_{11'}&y_{11'}+\hat{y}_{11'}& 0&1& -(y_{01'}+\hat{y}_{01'})\\
     y_{10'}+\hat{y}_{10'}&y_{11'}+\hat{y}_{11'}&0&0&0&1\end{bmatrix}
 \end{split}
 \end{equation}
with
 \begin{equation*}
 \begin{split}
 \widetilde{t}_{00'}&=-\frac{1}{2}\langle\mathbf{y}_{0'},\mathbf{y}_{0'}\rangle-y_{10'}\hat{y}_{00'}-\frac{1}{2}\langle\mathbf{\hat{y}}_{0'},\mathbf{\hat{y}}_{0'}\rangle-y_{00'}\hat{y}_{10'}
 =-\frac{1}{2}\langle\mathbf{y}_{0'}+\mathbf{\hat{y}}_{0'},\mathbf{y}_{0'}+\mathbf{\hat{y}}_{0'}\rangle,\\
 \widetilde{t}_{10'}&=\frac{1}{2}t-\frac{1}{2}\langle\mathbf{y}_{0'},\mathbf{y}_{1'}\rangle-\hat{y}_{00'}y_{11'}+\frac{1}{2}\hat{t}-\frac{1}{2}\langle\mathbf{\hat{y}}_{0'},\mathbf{\hat{y}}_{1'}\rangle-y_{01'}\hat{y}_{10'}\\&
 =\frac{1}{2}t+\frac{1}{2}\hat{t}-\frac{1}{2}\langle\mathbf{\hat{y}}_{0'},\mathbf{y}_{1'}
 \rangle+\frac{1}{2}\langle\mathbf{y}_{0'},\mathbf{\hat{y}}_{1'}\rangle-\frac{1}{2}\langle\mathbf{y}_{0'}
 +\mathbf{\hat{y}}_{0'},\mathbf{y}_{1'}+\mathbf{\hat{y}}_{1'}\rangle,\end{split} \end{equation*}
     \begin{equation*}
     \begin{split}
 \widetilde{t}_{01'}&=-\frac{1}{2}t-\frac{1}{2}\langle\mathbf{y}_{0'},\mathbf{y}_{1'}\rangle-y_{10'}\hat{y}_{01'}-\frac{1}{2}\hat{t}-\frac{1}{2}\langle\mathbf{\hat{y}}_{0'},\mathbf{\hat{y}}_{1'}\rangle-y_{00'}\hat{y}_{11'}\\&
 =-\frac{1}{2}t-\frac{1}{2}\hat{t}+\frac{1}{2}\langle\mathbf{\hat{y}}_{0'},\mathbf{y}_{1'}\rangle-\frac{1}{2}\langle\mathbf{y}_{0'},\mathbf{\hat{y}}_{1'}\rangle-\frac{1}{2}\langle\mathbf{y}_{0'}+\mathbf{\hat{y}}_{0'},\mathbf{y}_{1'}+\mathbf{\hat{y}}_{1'}\rangle,\\
 \widetilde{t}_{11'}& =-\frac{1}{2}\langle\mathbf{y}_{1'},\mathbf{y}_{1'}\rangle-y_{11'}\hat{y}_{11'}-\frac{1}{2}\langle\mathbf{\hat{y}}_{1'},\mathbf{\hat{y}}_{1'}\rangle-y_{01'}\hat{y}_{11'}
 =-\frac{1}{2}\langle\mathbf{y}_{1'}+\mathbf{\hat{y}}_{1'},\mathbf{y}_{1'}+\mathbf{\hat{y}}_{1'}\rangle.
 \end{split}
 \end{equation*}
It follows that
  $H_{(\mathbf{y},t )}\cdot H_{(\mathbf{\hat{y}},\hat{t})}
 =H_{(\mathbf{y}+\mathbf{\hat{y}},t+\hat{t}-\langle\mathbf{\hat{y}}_{0'},
 \mathbf{y}_{1'}\rangle+\langle\mathbf{y}_{0'},\mathbf{\hat{y}}_{1'}\rangle)},$
 i.e. $H$ is a group homomorphism. Obviously, it is an isomorphism.
\end{proof}

 Denote \begin{equation*}
 P_{\zeta}:=\begin{bmatrix}
     1&0&0&0&0&0\\
     \zeta&1&0&0&0&0\\
     0&0&1&0&0&0\\
     0&0&0&1&-\zeta&0\\
     0&0&0&0&1&0\\
     0&0&0&0&0&1\end{bmatrix}\in {\rm SO}(6,\mathbb{C}).\\
 \end{equation*}
 In fact, it's direct to check that
 $P_{\zeta}^{t}IP_{\zeta}=I.$
 We consider the two coordinate charts of $G/R$ as follows.
 $
 \varphi_{1}: \mathscr H\times\mathbb{C} \longrightarrow G/R$, $
 (\mathbf{y},t,\zeta) \longmapsto H_{(\mathbf{y},t)}\cdot P_{\zeta}\cdot R,$ and $
 \varphi_{2}: \mathscr H\times\mathbb{C} \longrightarrow G/R$, $
 (\mathbf{y},t,\zeta) \longmapsto  H_{(\mathbf{y},t)}\cdot W\cdot P_{\zeta}\cdot  R,
 $
where $W$ is the   element \begin{equation*}
 W:=\begin{bmatrix}
     0&1&0&0&0&0\\
     1&0&0&0&0&0\\
     0&0&1&0&0&0\\
     0&0&0&0&1&0\\
     0&0&0&1&0&0\\
     0&0&0&0&0&1\end{bmatrix}.
 \end{equation*}
 \begin{prop}
 The transition function of these two coordinates  charts is
 $
  \varphi_{1}^{-1}\circ\varphi_{2} : \mathscr H\times\mathbb{C^{*}}$ $ \longrightarrow \mathscr H\times\mathbb{C^{*}}$, $
 (\mathbf{y},t,\zeta) \mapsto(\mathbf{y},t,\zeta^{-1}).
 $
And $\mathscr H \times \mathbb{C}P^1$ is biholomorphically embeddded in ${\rm SO}(6,\mathbb{C})/R$.
 \end{prop}
 \begin{proof}
 Note that
 \begin{equation}\label{eq:decom}
  \begin{split}
  W\cdot P_{\zeta}
  &=\begin{bmatrix}
     \zeta&1&0&0&0&0\\
     1&0&0&0&0&0\\
     0&0&1&0&0&0\\
     0&0&0&0&1&0\\
     0&0&0&1&-\zeta&0\\
     0&0&0&0&0&1\end{bmatrix}=\begin{bmatrix}
     1&0&0&0&0&0& \\
     \zeta^{-1}&1&0&0&0&0\\
     0&0&1&0&0&0\\
     0&0&0&1&-\zeta^{-1}&0\\
     0&0&0&0&1&0\\
     0&0&0&0&0&1\end{bmatrix}\cdot
  \begin{bmatrix}
     \zeta&1&0&0&0&0\\
     0&-\zeta^{-1}&0&0&0&0\\
     0&0&1&0&0&0\\
     0&0&0&\zeta^{-1}&0&0\\
     0&0&0&1&-\zeta&0\\
     0&0&0&0&0&1\end{bmatrix}\\
  &
  :=P_{\zeta^{-1}}\cdot A_{\zeta}.
      \end{split}
  \end{equation}
  Then
  \begin{equation*}\begin{split}
 \varphi_{2}(\mathbf{y},t,\zeta)&= H_{(\mathbf{y},t)}\cdot W\cdot P_{\zeta}\cdot R
= H_{(\mathbf{y},t)}\cdot P_{\zeta^{-1}}\cdot A_{\zeta}\cdot R= H_{(\mathbf{y},t)}\cdot P_{\zeta^{-1}}\cdot R,
  \end{split}\end{equation*}
since $A_{\zeta}\in R$.
So  we get the transition function
$\varphi_{1}^{-1}\circ\varphi_{2}(\mathbf{y},t,\zeta)
  =\varphi_{1}^{-1}(H_{(\mathbf{y},t)}\cdot P_{\zeta^{-1}}\cdot R)
  =(\mathbf{y},t,\zeta^{-1}).$
  The proposition is proved.
 \end{proof}

 Similarly, we consider two coordinate charts of $G/Q$. One piece is
 $
 \psi_{1}: \mathbb{C}^{4} \longrightarrow G/Q $ given by $
 (\mathbf{x},t,\zeta)\longmapsto P_{\zeta}\cdot N_{(\mathbf{x},t)}\cdot Q,
 $
 where $\mathbf{x}=(x_{1},x_{2})\in\mathbb{C}^{2}$ and $N_{(\mathbf{x},t)}=H_{(\mathbf{y},t)}$ with $\mathbf{y}=\left[\begin{array} {cc} x_{1}&0\\x_{2}&0\end{array}\right]$.
The other is
 $
 \psi_{2}: \mathbb{C}^{4} \longrightarrow G/Q$ given by $
 (\mathbf{x},t,\zeta)\longmapsto W\cdot P_{\zeta}\cdot N_{(\mathbf{x},t)}\cdot Q.
$
\begin{prop}
 The transition function $\psi_{1}^{-1}\circ\psi_{2} : \mathbb{C}^{4}\setminus(\mathbb{C}^{3}\times\{0\})\longrightarrow \mathbb{C}^{4}\setminus(\mathbb{C}^{3}\times\{0\})$ of these two coordinates charts is given by
 $
 (\mathbf{x},t,\zeta)\mapsto(\zeta^{-1}\mathbf{x},-t-\zeta^{-1}x_{1}x_{2},\zeta^{-1}),
$
which coincides with (\ref{transi1}) up to a sign.
 \end{prop}
\begin{proof}
Note that
\begin{equation}\label{eq:you}
\begin{split}
\psi_{2}(\mathbf{x},t,\zeta)
&=W\cdot P_{\zeta}\cdot N_{(\mathbf{x},t)}\cdot Q
=P_{\zeta^{-1}}\cdot A_{\zeta}\cdot H_{(x_{1},x_{2},0,0,t)}\cdot A_{\zeta}^{-1}\cdot  Q,
\end{split}
\end{equation}
by (\ref{eq:decom}) and $A_{\zeta}\in Q.$ Note that
\begin{equation*}\label{eq:bo}
\begin{split}
&A_{\zeta}\cdot H_{(x_{1},x_{2},0,0,t)}\cdot A_{\zeta}^{-1}\\
=&\begin{bmatrix}
     \zeta&1&0&0&0&0& \\
     0&-\zeta^{-1}&0&0&0&0& \\
     0&0&1&0&0&0& \\
     0&0&0&\zeta^{-1}&0&0&\\
     0&0&0&1&-\zeta& 0 &\\
     0&0&0&0&0&1\end{bmatrix}\cdot
     \begin{bmatrix}
     1&0&0&0&0&0& \\
     0&1&0&0&0&0& \\
     x_{1}&0&1&0&0&0& \\
     -x_{1}x_{2}&-\frac{1}{2}t&-x_{2}&1&0&-x_{1}&\\
     \frac{1}{2}t&0&0&0&1& 0 &\\
     x_{2}&0&0&0&0&1\end{bmatrix}\cdot
      A_{\zeta}^{-1}\\
=&\begin{bmatrix}
     \zeta&1&0&0&0&0& \\
     0&-\zeta^{-1}&0&0&0&0& \\
     x_{1}&0&1&0&0&0& \\
     -\zeta^{-1}x_{1}x_{2}&-\frac{1}{2}\zeta^{-1}t&-\zeta^{-1}x_{2}&\zeta^{-1}&0&-\zeta^{-1}x_{1}&\\
     -x_{1}x_{2}-\frac{1}{2}t\zeta&-\frac{1}{2}t&-x_{2}&1&-\zeta&-x_{1}&\\
     x_{2}&0&0&0&0&1\end{bmatrix}\cdot
     \begin{bmatrix}\begin{smallmatrix}
     \zeta^{-1}&1&0&0&0&0& \\
     0&-\zeta&0&0&0&0& \\
     0&0&1&0&0&0& \\
     0&0&0&\zeta&0&0&\\
     0&0&0&1&-\zeta^{-1}&0&\\
     0&0&0&0&0&1\end{smallmatrix}\end{bmatrix}\end{split}
\end{equation*}
\begin{equation*}
\begin{split}=&\begin{bmatrix}
     1&0&0&0&0&0& \\
     0&1&0&0&0&0& \\
     \zeta^{-1}x_{1}&x_{1}&1&0&0&0& \\
     -\zeta^{-2}x_{1}x_{2}&\frac{1}{2}t-\zeta^{-1}x_{1}x_{2}&-\zeta^{-1}x_{2}&1&0&-\zeta^{-1}x_{1}&\\
     -\frac{1}{2}t-\zeta^{-1}x_{1}x_{2}&-x_{1}x_{2}&-x_{2}&0&1&-x_{1}&\\
    \zeta^{-1}x_{2}&x_{2}&0&0&0&1\end{bmatrix}\\
=&H_{(\zeta^{-1}x_{1},\zeta^{-1}x_{2},x_{1},x_{2},-t)}\\
=&H_{(\zeta^{-1}x_{1},\zeta^{-1}x_{2},0,0,-t-2\zeta^{-1}x_{1}x_{2})}\cdot H_{(0,0,x_{1},x_{2},0)}.
\end{split}
\end{equation*}

Substituting the above identity to (\ref{eq:you}) to get
\begin{equation*}
\psi_{2}(x_{1},x_{2},t,\zeta)
=P_{\zeta^{-1}}\cdot H_{(\zeta^{-1}x_{1},\zeta^{-1}x_{2},0,0,-t-2\zeta^{-1}x_{1}x_{2})}\cdot Q
=P_{\zeta^{-1}}\cdot N_{(\zeta^{-1}x_{1},\zeta^{-1}x_{2},-t-2\zeta^{-1}x_{1}x_{2})}Q,
\end{equation*}
by $H_{(0,0,x_{1},x_{2},0)}$$ \in Q$. So we have
\begin{equation*}
  \psi_{1}^{-1}\circ\psi_{2}(x_{1},x_{2},t,\zeta)
  =\psi_{1}^{-1}(P_{\zeta^{-1}}\cdot N_{(\zeta^{-1}x_{1},\zeta^{-1}x_{2},-t-2\zeta^{-1}x_{1}x_{2})}Q)
  =(\zeta^{-1}x_{1},\zeta^{-1}x_{2},-t-2\zeta^{-1}x_{1}x_{2},\zeta^{-1})
\end{equation*}
by  definition. The proposition is proved.
\end{proof}

\begin{lem}\label{eq:P1}
$P_{\zeta}^{-1}H_{(\mathbf{y}_{0'},\mathbf{y}_{1'},t)}P_{\zeta}=H_{(\mathbf{y}_{0'}+\zeta\mathbf{y}_{1'},\mathbf{y}_{1'},t)}.$
\end{lem}
\begin{proof} We have
\begin{equation*}
\begin{split}&
P_{\zeta}^{-1}H_{(\mathbf{y}_{0'},\mathbf{y}_{1'},t)}P_{\zeta} \\
=&\begin{bmatrix}\begin{smallmatrix}
     1&0&0&0&0&0& \\
     -\zeta&1&0&0&0&0& \\
     0&0&1&0&0&0& \\
     0&0&0&1&\zeta&0&\\
     0&0&0&0&1& 0 &\\
     0&0&0&0&0&1
     \end{smallmatrix}\end{bmatrix}\cdot
     \begin{bmatrix}\begin{smallmatrix}
     1&0&0&0&0&0& \\
     0&1&0&0&0&0& \\
     y_{00'}&y_{01'}&1&0&0&0&\\
     -\frac{1}{2}\langle\mathbf{y}_{0'},\mathbf{y}_{0'}\rangle&-\frac{1}{2}t-\frac{1}{2}\langle\mathbf{y}_{0'},\mathbf{y}_{1'}
     \rangle&-y_{10'}& 1&0& -y_{00'}&\\
     \frac{1}{2}t-\frac{1}{2}\langle\mathbf{y}_{0'},\mathbf{y}_{1'}\rangle&-\frac{1}{2}\langle\mathbf{y}_{1'},\mathbf{y}_{1'}
     \rangle&-y_{1'1'}&0&1& -y_{01'} &\\
     y_{10'}&y_{11'}&0&0&0&1\end{smallmatrix}\end{bmatrix}\cdot P_{\zeta}\\
=&\begin{bmatrix}\begin{smallmatrix}
     1&0&0&0&0&0& \\
     -\zeta&1&0&0&0&0& \\
     y_{00'}&y_{01'}&1&0&0&0&\\
     -\frac{1}{2}\langle\mathbf{y}_{0'},\mathbf{y}_{0'}\rangle+\frac{1}{2}t\zeta-\frac{1}{2}\zeta\langle\mathbf{y}_{0'},\mathbf{y}_{1'}\rangle&-\frac{1}{2}t-\frac{1}{2}\langle\mathbf{y}_{0'},\mathbf{y}_{1'}\rangle
     -\zeta\langle\mathbf{y}_{1'},\mathbf{y}_{1'}\rangle&-y_{10'}-\zeta y_{11'}& 1&\zeta& -y_{00'}-\zeta y_{01'}&\\
     \frac{1}{2}t-\frac{1}{2}\langle\mathbf{y}_{0'},\mathbf{y}_{1'}\rangle&-\frac{1}{2}\langle\mathbf{y}_{1'},\mathbf{y}_{1'}\rangle&-y_{1'1'}&0&1& -y_{01'} &\\
     y_{10'}&y_{11'}&0&0&0&1\end{smallmatrix}\end{bmatrix}\cdot P_{\zeta}\\
=&\begin{bmatrix}\begin{smallmatrix}
     1&0&0&0&0&0& \\
     0&1&0&0&0&0& \\
     y_{00'}+\zeta y_{01'}&y_{01'}&1&0&0&0&\\
     -\frac{1}{2}\langle\mathbf{y}_{0'}+\zeta \mathbf{y}_{1'},\mathbf{y}_{0'}+\zeta \mathbf{y}_{1'}\rangle&-\frac{1}{2}t-\frac{1}{2}\langle\mathbf{y}_{0'}+\zeta \mathbf{y}_{1'},\mathbf{y}_{1'}\rangle&-y_{10'}-\zeta y_{11'}& 1&0& -y_{00'}-\zeta y_{01'}&\\
     \frac{1}{2}t-\frac{1}{2}\langle\mathbf{y}_{0'}+\zeta \mathbf{y}_{1'},\mathbf{y}_{1'}\rangle&-\frac{1}{2}\langle\mathbf{y}_{1'},\mathbf{y}_{1'}\rangle&-y_{1'1'}&0&1& -y_{01'} &\\
     y_{10'}+\zeta y_{11'}&y_{11'}&0&0&0&1\end{smallmatrix}\end{bmatrix}\\&
      =H_{(\mathbf{y}_{0'}+\zeta \mathbf{y}_{1'},\mathbf{y}_{1'},t)}
\end{split}
\end{equation*}
by the definition (\ref{eq:Hisen}) of $H_{(\mathbf{y}_{0'},\mathbf{y}_{1'},t)}$.
\end{proof}
 \begin{prop}\label{ma-pen} The mapping $\eta:{\rm SO}(6,\mathbb{C})/R\longrightarrow {\rm SO}(6,\mathbb{C})/Q$ in (\ref{Ptwistor1}) is locally given by
 \begin{equation*}
 \eta\left(H_{(\mathbf{y},t)} \cdot P_{\zeta}R\right)=P_{\zeta}\cdot N_{\left(y_{00'}+\zeta y_{01'},y_{10'}+\zeta y_{11'},t-\langle\mathbf{y}_{0'}+\zeta \mathbf{y}_{1'},\mathbf{y}_{1'}\rangle\right)}Q,
 \end{equation*}
which coincides with (\ref{eq:psi}).
 \end{prop}
\begin{proof}
By (\ref{eq:P1}), we have
 \begin{equation*}
 \begin{split}&
   \eta\left(H_{(\mathbf{y},t)} \cdot P_{\zeta}R\right)
   = H_{(\mathbf{y}_{0'},\mathbf{y}_{1'},t)} \cdot P_{\zeta} Q
   =P_{\zeta}\cdot P_{\zeta}^{-1}\cdot H_{(\mathbf{y}_{0'},\mathbf{y}_{1'},t)} \cdot P_{\zeta} Q
   =P_{\zeta}\cdot H_{(\mathbf{y}_{0'}+\zeta\mathbf{y}_{1'},\mathbf{y}_{1'},t)} Q\\&
   =P_{\zeta}\cdot H_{(\mathbf{y}_{0'}+\zeta\mathbf{y}_{1'},\mathbf{0},t-\langle\mathbf{y}_{0'}+\zeta \mathbf{y}_{1'},\mathbf{y}_{1'}\rangle)}\cdot H_{(\mathbf{0},\mathbf{y}_{1'},0)}Q
   =P_{\zeta}\cdot H_{(\mathbf{y}_{0'}+\zeta\mathbf{y}_{1'},\mathbf{0},t-\langle\mathbf{y}_{0'}+\zeta \mathbf{y}_{1'},\mathbf{y}_{1'}\rangle)}Q\\&
    =P_{\zeta}\cdot N_{(y_{00'}+\zeta y_{01'},y_{10'}+\zeta y_{11'},t-\langle\mathbf{y}_{0'}+\zeta \mathbf{y}_{1'},\mathbf{y}_{1'}\rangle)}Q
  \end{split}
 \end{equation*}
 by $H_{(\mathbf{0},\mathbf{y}_{1'},0)}\in Q$. The result follows.\end{proof}

\end{document}